\documentclass[reqno]{amsart}

\usepackage{xypic,hyperref,todonotes,comment}
\usepackage{graphics,amssymb}
\usepackage{pagecolor,lipsum}
\usepackage{latexsym}
\usepackage{mathrsfs}
\xyoption{curve}
\usepackage{hyperref}
\hypersetup{ 
colorlinks,
citecolor=violet,
filecolor=blue,
linkcolor=black,
urlcolor=magenta
}

\def\longmapsto#1{%
\begin{tikzpicture}
\draw (0,0.5mm) -- (0,-0.5mm);
\newlength\mylength
\setlength{\mylength}{\widthof{\hspace{.5mm}#1\hspace{.5mm}}}
\draw[->] (0,0) -- (1.2\mylength,0) node[above,midway] {#1};
\end{tikzpicture}
}

\newtheorem{lemma}{Lemma}[section]
\newtheorem{proposition}[lemma]{Proposition}
\newtheorem{theorem}[lemma]{Theorem}
\newtheorem{corollary}[lemma]{Corollary}

\newtheorem{conjecture}[lemma]{Conjecture}

\newtheorem*{propositionA}{Proposition}

\newtheorem*{theoremA}{Theorem}
\newtheorem*{corollaryA}{Corollary}

\theoremstyle{definition}

\newtheorem{definition}[lemma]{Definition}
\newtheorem{remark}[lemma]{Remark}
\newtheorem*{definitionA}{Definition}

\newcommand{\mfk}[1]{\mathfrak{#1}}
\newcommand{\mbb}[1]{\mathbb{#1}}
\newcommand{\mcl}[1]{\mathcal{#1}}
\newcommand{\mrm}[1]{\mathrm{#1}}
\newcommand{\msc}[1]{\mathscr{#1}}
\newcommand{\mbf}[1]{\mathbf{#1}}
\newcommand{\opn}[1]{\operatorname{#1}}

\newcommand{\ot}{\otimes}
\newcommand{\bt}{\boxtimes}
\newcommand{\C}{\mathscr{C}}
\newcommand{\D}{\mathscr{D}}
\newcommand{\M}{\mathscr{M}}
\newcommand{\g}{\mathfrak{g}}
\renewcommand{\1}{\mathbf{1}}
\renewcommand{\1}{\mathbf{1}}
\renewcommand{\O}{\mathscr{O}}
\renewcommand{\b}{\text{\tiny$\bullet$}}

\DeclareMathOperator{\Hom}{Hom}
\DeclareMathOperator{\End}{End}

\DeclareMathOperator{\Ext}{Ext}
\DeclareMathOperator{\Aut}{Aut}
\DeclareMathOperator{\Sym}{Sym}
\DeclareMathOperator{\rep}{rep}
\DeclareMathOperator{\Kdim}{Kdim}
\DeclareMathOperator{\GKdim}{GKdim}
\DeclareMathOperator{\FPdim}{FPdim}

\title[]{Cohomology of finite tensor categories: duality and Drinfeld centers}
\date{\today}

\author{Cris Negron}
\thanks{The first author was supported by NSF Postdoctoral Research Fellowship DMS-1503147.  The second author was supported by NSF grant DMS-1802503.  The second author was also supported by an AMS Simons travel grant.}
\email{negronc@mit.edu}
\address{Department of Mathematics\\Massachusetts Institute of Technology\\
Cambridge, MA USA}

\author{Julia Yael Plavnik}
\email{jplavnik@iu.edu}
\address{Department of Mathematics\\Indiana University\\ 
Bloomington IN, USA}

\begin{document}

\maketitle

\begin{abstract}
We consider the finite generation property for cohomology of a finite tensor category $\msc{C}$, which requires that the self-extension algebra of the unit $\Ext^\b_\msc{C}(\mbf{1},\mbf{1})$ is a finitely generated algebra and that, for each object $V$ in $\msc{C}$, the graded extension group $\Ext^\b_\msc{C}(\mbf{1},V)$ is a finitely generated module over the aforementioned algebra.  We prove that this cohomological finiteness property is preserved under duality (with respect to exact module categories) and taking the Drinfeld center, under suitable restrictions on $\msc{C}$.  For example, the stated result holds when $\msc{C}$ is a braided tensor category of odd Frobenius-Perron dimension.  By applying our general results, we obtain a number of new examples of finite tensor categories with finitely generated cohomology.  In characteristic $0$, we show that dynamical quantum groups at roots of unity have finitely generated cohomology.  We also provide a new class of examples in finite characteristic which are constructed via infinitesimal group schemes.
\end{abstract}

\section{Introduction}

Fix a base field $k$.  This field is arbitrary in general, but at times specified to be perfect or of characteristic zero. For a given finite tensor category $\msc{C}$ over $k$, by the cohomology $H^\b(\msc{C},V)$ we mean the graded group of extensions from the unit $H^\b(\msc{C},V):=\Ext^\b_\msc{C}(\1,V)$.  For us, a {\it tensor category} always means a {\it finite tensor category over the base field $k$}.  We are concerned with the following cohomological finiteness property:

\begin{definitionA}
We say a tensor category $\msc{C}$ is of {\it finite type} (over $k$) if its cohomo\-logy $H^\b(\msc{C},\1)$ is a finitely generated algebra and $H^\b(\mathscr{C},V)$ is a finitely generated $H^\b(\mathscr{C},\1)$-module for each $V$ in $\msc{C}$.  For $\msc{C}$ of finite type, we define the Krull dimension of $\msc{C}$ as $\Kdim\msc{C}:=\Kdim H^{ev}(\msc{C},\1)$.
\end{definitionA}

It is conjectured in~\cite{etingofostrik04} that any finite tensor category is of finite type (see also~\cite[pg 2]{friedlandersuslin97}).  We refer to this conjecture as the finite generation conjecture.  For a short historical account of this conjecture, including myriad known examples and references, one can see the introduction to~\cite{friedlandernegron18}.
\par

With the present study we would like to initiate an investigation into the general behaviors of cohomology, as a tensor invariant.  In particular, we are interested in how cohomology evolves under certain standard operations, such as de-/equivariantization, extension, and duality.  This project has a strong intersection with the finite generation conjecture, but is more local in nature, as one needn't draw any global conclusion about the nature of cohomology in order to make explicit relative statements.  (One could compare with theories of support, e.g.~\cite[Corollary 5.4, Proposition 7.1]{suslinfriedlanderbendel97}.)  We conjecture the following.

\begin{conjecture}[Cohomological stability]\label{conj:stab_conj}
If $\msc{C}$ is a tensor category which is of finite type and $\msc{M}$ is an exact $\msc{C}$-module category, then the corresponding dual category $\msc{C}^\ast_\msc{M}$ is also of finite type.  Furthermore, in characteristic $0$, the Krull dimension is invariant under duality
\[
\Kdim\msc{C}^\ast_\msc{M}=\Kdim\msc{C}.
\] 
\end{conjecture}

In general, it is quite difficult to establish a precise equality of Krull dimensions under duality.  So here we focus on uniformly {\it bounding} the Krull dimensions of the duals $\msc{C}^\ast_\msc{M}$ by a polynomial function in $\Kdim\msc{C}$.  By {\it uniform} we mean uniform across all $\msc{M}$.  Whence we have a relaxation of the above conjecture.

\begin{conjecture}[Weak cohomological stability]\label{conj:weak_stab}
The finite type property is preserved under categorical duality, as in Conjecture~\ref{conj:stab_conj}.  Furthermore, in characteristic $0$, there is a polynomial $P_\mrm{univ}\in \mbb{R}_{\geq 0}[X]$ which provides a uniform bound on the Krull dimensions of the duals
\[
\Kdim\msc{C}^\ast_\msc{M}\leq P_\mrm{univ}(\Kdim\msc{C}).
\]
\end{conjecture}

In Conjecture~\ref{conj:weak_stab} the polynomial $P_\mrm{univ}$ is intended to be independent of choice of $\msc{C}$ and $\msc{M}$.  We note that Conjecture~\ref{conj:weak_stab} also implies a lower bound on the Krull dimensions of the duals, as $P_\mrm{univ}$ is an increasing function in $X$ and is therefore invertible.  We remark here that a precise understanding of the behavior of the Krull dimension under duality would be especially interesting in light of the fact that more refined properties of the spectrum of cohomology are decidedly {\it not} preserved under such operations (see Remark~\ref{rem:q_group}).
\par

The notion of categorial duality, also known as categorical Morita equivalence, was introduced by M\"uger in~\cite{muger03I} (cf.~\cite{ostrik03}) and now plays a fundmental role in the general theory of tensor categories.  The precise construction of the dual $\msc{C}^\ast_\msc{M}$ is recalled in Section~\ref{sect:intro_duals} and a discussion of the specific meaning of Conjectures~\ref{conj:stab_conj} and~\ref{conj:weak_stab} for Hopf algebras is given below, in Subsection \ref{stability of cohomology-Hopf}.
\par

In the present paper we prove the weak stability conjecture for ``most" braided tensor categories in characteristic $0$; specifically braided tensor categories with semisimple M\"uger center.  The polynomial in this case is simply $P_\mrm{braid}(X)=2X$.  More explicit descriptions of our main results and examples are given below.
\par

Our general method is to prove that the Drinfeld center $Z(\msc{C})$ of $\msc{C}$ is of finite type provided that $\msc{C}$ is of finite type, then to appeal to the fact that the construction of the center is invariant under categorical duality (see Proposition~\ref{prop:conjs}/\ref{prop:wconjs}).  The Drinfeld center approach allows us to obtain some results for categories in finite characteristic as well.  In terms of the center, the stability conjecture proposes that $Z(\msc{C})$ is of finite type whenever $\msc{C}$ is of finite type and that $\Kdim Z(\msc{C})=2\Kdim\msc{C}$.  Similarly, weak stability proposes that the finite type property is preserved under formation of the center and that the Krull dimension of $Z(\msc{C})$ grows sub-exponentially as a function of $\Kdim\msc{C}$.

\begin{remark}
Formally speaking, the dual $\msc{C}^\ast_\msc{M}$ will be a ``multi-tensor" category when $\msc{M}$ is decomposable.  Although we are not interested in this case, decomposability of $\msc{M}$ causes no problems for us, so we allow $\msc{M}$ to be decomposable in general.
\end{remark}
\begin{remark}
The proposed bound on the Krull dimension is clearly false in finite characteristic.  For example, $\rep(\mbb{F}_p[\mbb{Z}/p\mbb{Z}]^{\ot d})$ has Krull dimension $d$ while $\rep\left(\O(\mbb{Z}/p\mbb{Z})^{\ot d}\right)$ has Krull dimension $0$, even though $\rep(\O(\mbb{Z}/p\mbb{Z})^{\ot d})^\ast_{Vect}$ is tensor equivalent to $\rep(\mbb{F}_p[\mbb{Z}/p\mbb{Z}]^{\ot d})$.
\end{remark}

\subsection{Stability of cohomology for Hopf algebras}\label{stability of cohomology-Hopf}
Let us explain the situation a bit more clearly for the Hopf algebraically inclined.  As with categories, a {\it Hopf algebra} always means a {\it finite dimensional Hopf algebra}.
\par

In the Hopf setting, Conjecture~\ref{conj:stab_conj} can be seen as a generalization of a theorem of Larson and Radford, which states that a finite dimensional Hopf algebra in characteristic $0$ is semisimple if and only if its dual is semisimple~\cite{larsonradford88}.  

Let us consider some illustrative examples.  Take a Drinfeld twist $J$ of a Hopf algebra $A$.  Then we have a canonical tensor equivalence $\rep(A)\overset{\sim}\to \rep(A^J)$, where $A^J$ is obtained by altering the comultiplication of $A$ via $J$.  We also have the forgetful functor $\rep(A)\to Vect$ and the alternate functor $\rep(A)\to \rep(A^J)\to Vect$.  These functors produce $\mrm{rep}(A)$-module categories $\msc{M}=Vect$ and $Vect_J$, where $Vect_J$ is $Vect$ with the $\rep(A)$-action ``twisted" by $J$.  The dual categories are
\[
\rep(A)_{Vect}^\ast=\rep(A^\ast)^{\mrm{cop}},\ \ \rep(A)_{Vect_J}^\ast=\rep((A^\ast)_J)^{\mrm{cop}},
\]
where in the second instance $J$ alters the multiplicative structure on $A^\ast$ as a cocycle twist.  The superscript ``$\mrm{cop}$" here means we are taking the opposite tensor product.  (Notice that the $\mrm{cop}$ operation does not affect cohomology.)  One can also find a module category $\msc{M}(\sigma)$ associated to any $2$-cocycle $\sigma:A\otimes A\to k$ so that $\rep(A)^\ast_{\msc{M}(\sigma)}\cong \rep(A_\sigma)$.
\par

Whence the (weak) stability conjecture proposes, among other things, that if $A$ has finitely generated cohomology then any cocycle twist $A_\sigma$ also has finitely generated cohomology, as does its dual $A^\ast$ and any cocycle twists of its dual $A^\ast_J$.  Furthermore, the Krull dimension of cohomology is proposed to be invariant under these operations, or to vary at most as a polynomial in $\Kdim H^{ev}(A,k)$.

\subsection{New examples}

Among our contributions herein, there are two explicit classes of new examples of finite type tensor categories.  In finite characteristic, we show that if $G$ is a Frobenius kernel in a smooth algebraic group $\mbb{G}$, $G=\mbb{G}_{(r)}$, then arbitrary duals of the category of representations $\rep(G)^\ast_\msc{M}$ are of finite type.  Furthermore, we uniformly bound the Krull dimensions
\[
\Kdim \rep(G)^\ast_\msc{M}\leq \Kdim \rep(G)+\dim\mfk{g},
\]
where $\mfk{g}$ is the Lie algebra of $G$.  These results are obtained as an application of our work here in conjunction with results of Friedlander and the first author~\cite{friedlandernegron18}, and appear in Corollary~\ref{cor:repOM} below.  We note that exact module categories over $\rep(G)$ were classified via cohomological data by Gelaki~\cite{gelaki15}.
\par

One can produce concrete examples of new Hopf algebras with finitely gene\-rated cohomology via Corollary~\ref{cor:repOM}.  Specifically, one can take cocycle twists of the function algebra $\O(G)$ to produce Hopf algebras in characteristic $p$ which are neither commutative nor cocommutative but are seen to have finitely generated cohomology (see Section~\ref{sect:O_sig}).
\par

In characteristic 0, we prove that the dynamical quantum groups of Etingof and Nikshych~\cite{etingofnikshych01} (see also~\cite{etingofvarchenko99,xu01}) have finitely generated cohomology.  More in depth descriptions of dynamical quantum groups are given in Section \ref{sect:uqT} and Appendix~\ref{sect:A} but let us say here that while usual quantum groups are associated to constant solutions to the Yang-Baxter equation, dynamical quantum groups are associated to parameter dependent solutions to the Yang-Baxter equation.  In particular, dynamical quantum groups are not Hopf algebras in the strict sense of the term, although they do have associated tensor categories of representations.  These examples are covered in Section~\ref{sect:uqT}.
\par

More generally, we consider the pointed Hopf algebras $u(\mfk{D})$ of Andruskiewitsch and Schneider~\cite{andruskiewitschschneider10}.  We apply results of Mastnak, Pevtsova, Schauenburg, and Witherspoon~\cite{mpsw10} to find that arbitrary duals $\rep(u(\mfk{D}))^\ast_\msc{M}$ are of finite type, and also bound their Krull dimensions.  This result appears in Theorem~\ref{thm:char0}.

\subsection{Description of main results}
We begin with the following

\begin{propositionA}[\ref{prop:fgcat}]
Suppose that $\msc{D}$ is a tensor category of finite type and that $F:\msc{D}\to \msc{C}$ is a surjective tensor functor.  Then $\msc{C}$ is also of finite type and has bounded Krull dimension $\Kdim\C\leq \Kdim\D$.
\end{propositionA}

The result is essentially a consequence of the fact that any surjective $F$ admits an exact right adjoint, and a similar argument is employed to obtain a finite generation property for exact module categories over finite type categories (Proposition~\ref{prop:M}).  When one considers representation categories of Hopf algebras, Proposition~\ref{prop:fgcat} appears as follows.

\begin{corollaryA}[\ref{cor:fgsubalg}]
Suppose that $A\to D$ is a Hopf inclusion and that $D$ has finitely generated cohomology.  Then $A$ also has finitely generated cohomology and the Krull dimension is bounded as $\Kdim H^{ev}(A,k)\leq \Kdim H^{ev}(D,k)$.
\end{corollaryA}

The center $Z(\msc{C})$ provides the necessary link between the cohomology of $\msc{C}$ and the cohomology of its duals $\msc{C}^\ast_\msc{M}$.  More specifically, any dual $\msc{C}^\ast_\msc{M}$ admits a surjective tensor functor from the center $Z(\msc{C})$~\cite[Proposition\ 8.5.3]{egno15}, and hence $\msc{C}^\ast_\msc{M}$ is of finite type and of Krull dimension $\leq \Kdim Z(\msc{C})$ whenever the center of $\msc{C}$ is of finite type.  We use Proposition~\ref{prop:fgcat} in the braided setting to address (weak) stability of cohomology.
\par

Recall that for a Hopf algebra $A$, the Drinfeld center of $\rep(A)$ is identified with representations of the Drinfeld double $D(A)$, and a braiding on $\rep(A)$ is exactly the information of a quasitriangular structure on $A$.  Recall also that the M\"uger center of a braided tensor category $\msc{C}$ is the full tensor subcategory of all objects $V$ in $\msc{C}$ which braid ``trivially" with every other object in $\msc{C}$ (see Section~\ref{sect:braidedI}).

\begin{theoremA}[\ref{thm:braidedII}]
Let $\msc{C}$ be a braided tensor category of finite type over a field of characteristic $0$. Suppose that the M\"uger center of $\msc{C}$ is semisimple.  Then any dual category $\msc{C}^\ast_\msc{M}$ with respect to an exact $\msc{C}$-module category $\msc{M}$ is of also of finite type.  Furthermore, there is a uniform bound on the Krull dimensions
\[
\Kdim\msc{C}^\ast_\msc{M}\leq 2\Kdim\msc{C}.
\]
\end{theoremA}

Basic information on the M\"uger center, as well as means of determining its semisimplicity, are given in Sections~\ref{sect:Mu},~\ref{sect:tann_check}, and~\ref{sect:ssE}.  Theorem~\ref{thm:braidedII} is more easily understood through a simple corollary.  Recall that a tensor category $\msc{C}$ is called weakly integral if its Frobenius-Perron dimension is an integer.  For example, representation categories of Hopf algebras are weakly integral with $\opn{FPdim}(\rep(A))=\dim A$.

\begin{corollaryA}[\ref{lem:4}]
Suppose $\msc{C}$ is of finite type over a field of characteristic $0$. Assume also that $\msc{C}$ is of integral Frobenius-Perron dimension which is not divisible by $4$.  If $\msc{C}$ admits a braiding, then for every exact module category $\msc{M}$ the corresponding dual $\msc{C}^\ast_\msc{M}$ is also of finite type and, furthermore, $\Kdim\msc{C}^\ast_\msc{M}\leq 2\Kdim\msc{C}$.
\end{corollaryA}

For non-degenerate categories, i.e.\ categories with M\"uger center equivalent to $Vect$, Theorem~\ref{thm:braidedII} is a fairly straightforward application of Proposition~\ref{prop:fgcat}.  We deal with this case independently in Proposition~\ref{prop:Cduals}.  When the M\"uger center of $\msc{C}$ is not trivial, the situation becomes much more dynamic.  (The problem is that the Drinfeld center construction is not functorial, and therefore becomes difficult to handle under de-equivariantization of the input.)  We must consider here two cases: the case where the M\"uger center is Tannakian and the case where the M\"uger center is non-Tannakian but semisimple.  We address the Tannakian case in Section~\ref{sect:braidedII} and the non-Tannakian case in Section~\ref{sect:braidedIII}.
\par

In terms of braided categories, there is a final possibility which we do not address here.  Namely, when the M\"uger center of $\msc{C}$ is the representation category of a super group with non-vanishing odd functions (see Section~\ref{sect:Mu}).  At the moment, it seems that this case will require either a different approach than the one taken here, or a much more non-trivial analysis of spectral sequences relating the cohomology of $\msc{C}$ to that of its Drinfeld center.

\subsection{General organization}

The present paper has two main portions.  In the first part, which consists of Sections~\ref{sect:HC}--\ref{sect:uqT}, we give results relating the finite type property for the center $Z(\msc{C})$ to the finite type property for arbitrary duals $\msc{C}^\ast_\msc{M}$.  These materials are punctuated by the examples of Sections~\ref{sect:O_sig} and~\ref{sect:uqT}.  For the first portion of the paper $k$ is a base field of arbitrary characteristic, unless explicitly stated otherwise.  In the second portion, which consists of Sections~\ref{sect:braidedI}--\ref{sect:braidedIII}, we pursue the (weak) stability conjecture for braided tensor categories more generally.  As relayed above, we provide a proof of Conjecture~\ref{conj:weak_stab} for braided tensor categories with semisimple M\"uger center.  We always assume $k$ is of characteristic $0$ in this latter portion of the paper.
\par

In Appendix~\ref{sect:A}, we discuss relations between dynamical twists and module categories.

\subsection*{Acknowledgement}

Thanks to Sarah Witherspoon for many helpful conversations.  Thanks also to Pavel Etingof, Shlomo Gelaki, Victor Ostrik, and Ingo Runkel for clarifying points in the literature, and to Peter Haine as well.  We extend special thanks to Eric Rowell, who provided a number of key ideas which were central to the completion of Section~\ref{sect:braidedIII}.

\tableofcontents

\section{Background on tensor categories and cohomology}

We assume the reader is well-acquainted with Hopf algebras.  We give here some information about the more general framework of tensor categories.
\par

\subsection{Conventions}
We fix a base field $k$ of arbitrary characteristic.  We let $\Kdim H^\b$ denote the Krull dimension of a graded commutative algebra $H^\b$, i.e.\ the supremum of the lengths of chains of prime ideals in $H^\b$.  Since $H^\b$ is graded commutative, any odd degree element is nilpotent.  So the Krull dimension of $H^\b$ is equal to the Krull dimension of its (commutative) even subalgebra $\Kdim H^\b=\Kdim H^{ev}$.
\par

By a ``$k$-linear category" we mean an {\it abelian} category enriched over $Vect$.  For a tensor category $\msc{C}$, we let $\msc{C}^\mrm{cop}$ denote the tensor category which is equal to $\msc{C}$ as a $k$-linear category along with the opposite tensor product $V\ot^\mrm{cop}W=W\ot V$.  For a Hopf algebra $A$, for example, $\mrm{rep}(A)^\mrm{cop}=\mrm{rep}(A^\mrm{cop})$.  When $\msc{C}$ is a braided tensor category, with braiding $c$, we write $\msc{C}^\mrm{rev}$ for the tensor category $\msc{C}$ equipped with the reverse braiding, $c^\mrm{rev}_{V,W}=c_{W,V}^{-1}$.
\par

Standard, categorial, opposites are denoted by a non-roman superscript, $\msc{C}^{op}$.  We let $G(\msc{C})$ denote the group of (ismorphism classes of) invertible objects in a tensor category $\msc{C}$.

\subsection{Tensor categories and representations of Hopf algebras}

A $k$-linear (abelian) category $\C$ is said to be {\it finite} if it is equivalent to the category of modules over a finite dimensional algebra.  Rather, $\C$ is finite if it has finitely many simple objects (up to isomorphism), finite dimensional hom spaces, enough projectives, and all objects have finite length.
\par

In this work, by a {\it tensor category} we will always mean a {\it finite tensor category}.

\begin{definition}[\cite{egno15}]
A (finite) tensor category $\C$ is a $k$-linear, finite, rigid, monoidal category such that the monoidal structure $\ot:\C\times \C\to \C$ is $k$-linear in each factor.  We require additionally that the unit object $\mbf{1}$ of $\C$ is simple and that $\End_\C(\1,\1)=k$.
\end{definition}

At times we label the unit of $\C$ as $\1_\C$.  However, when no confusion will arise, we employ the simpler notation $\1$.  Of course, such a $\msc{C}$ comes equipped with a (generally nontrivial) associator, and some additional compatibilities with the unit.  Although these seemingly subtle structures are quite important in general, they do not play a significant role in our study.
\par

The rigid structure refers to the existence of left and right duals for each object $V$ in $\C$.  These are objects $V^\ast$ and ${^\ast V}$ respectively which come equipped with evaluation
\[
ev^l_V:V^\ast\ot V\to \1,\ \ ev^r_V:V\ot {^\ast V}\to \1,
\]
and coevaluation
\[
coev^l_V:\1\to V\ot V^\ast,\ \ coev^r_V:\1\to {^\ast V}\ot V,
\]
maps which satisfy a number of exceedingly useful axioms.  We do not list the axioms here, but refer the reader to~\cite[Section\ 2.10]{egno15} for details and basic implications.
\par

The algebraically inclined reader is free to think of $\C$ as the category of (finite dimensional) representations of a finite dimensional Hopf algebra $A$.  In this case $\rep(A)$ has the usual monoidal structure $\ot=\ot_k$ and the duals are give by
\[
V^\ast=\Hom_k(V,k)\ \text{with $A$-action }a\cdot f=\left(v\mapsto f(S(a)v)\right),
\]
\[
{^\ast V}=\Hom_k(V,k)\ \text{with $A$-action }a\cdot f=\left(v\mapsto f(S^{-1}(a)v)\right).
\]
The evaluation maps are the usual ones, $f\ot v\mapsto f(v)$ and $v\ot f\mapsto f(v)$, respectively. If we choose dual bases $\{v_i,f_i\}_i$ for a representation and its linear dual then the coevaluation maps are given by $1\mapsto \sum_i v_i\ot f_i$ and $1\mapsto \sum_i f_i\ot v_i$,  respectively.

\subsection{The Yoneda-product on cohomology, a quick review}

Recall that the bounded derived category $\mrm{D}^b(\msc{C})$ of an abelain category $\mathscr{C}$ has objects given by complexes with bounded cohomology and morphisms given by equivalence classes of pairs $X\overset{s}\leftarrow Y\overset{f}\to X'$, where $s$ is a quasi-isomorphism.  We denote such an equivalence class by $fs^{-1}:X\to X'$.  The composition of two morphisms $X\leftarrow Y\to X'$ and $X'\leftarrow Y'\to X''$ is given by the top of any diagram
\[
\xymatrixrowsep{3mm}
\xymatrix{
 & & Y''\ar[dl]_t\ar[dr] & \\
 & Y\ar[dl]\ar@{-->}[dr] & & Y'\ar[dr]\ar@{-->}[dl] \\
X & & X' & & X'',
}
\]
where $t$ is a quasi-isomorphism.  We note that the above diagram occurs in the homotopy category of $\mathscr{C}$, not the category of chain complexes.
\par

One can readily check that composition of extensions of the unit object $\mbf{1}\to \Sigma^\ast \mbf{1}$ in $\mrm{D}^b(\C)$ is related to the tensor product by
\begin{equation}\label{eq:327}
(gt^{-1})\circ (fs^{-1})\ = \left(\xymatrix{
\1 & X\ot Y\ar[l]_{s\ot t}\ar[r]^{f\ot g} & \Sigma^{n+m}\1
}\right).
\end{equation}
So the natural concatenation of morphisms via the tensor product in $\mrm{D}^b(\C)$ is opposite to the composition operation.

\begin{remark}
Since the algebra $H^\b(\C,\1)=\oplus_n \Hom_{\mrm{D}^b(\C)}(\1,\Sigma^n\1)$ is graded commutative~\cite{suarez04}, the difference between $H^\b(\C,\1)$ and $H^\b(\C,\1)^{op}$ is essentially negligible, we will however keep track of the distinction for this subsection.
\end{remark}

We have for any $V$ in $\C$ the exact functor $-\ot V$ and subsequent algebra morphism
\[
-\ot V:H^\b(\C,\1)^{op}\to \Ext^\b_\C(V,V).
\]
The map $-\ot V$ takes an extension $\1\overset{s}\leftarrow X\overset{f}\to \Sigma^n\1$ of $1$ to the extension
\[
\xymatrix{
V & X\ot V\ar[l]_{s\ot V}\ar[r]^{f\ot V} & \Sigma^nV
}
\]
of $V$.  By an argument similar to the one employed in~\cite{suarez04} one finds

\begin{lemma}
The algebra map $-\ot V:H^{\b}(\msc{C},\mbf{1})^{op}\to \Ext^\b_\C(V,V)$ has image in the (graded) center of $\Ext^\b_\C(V,V)$.
\end{lemma}

We define a left action of $H^\b(\C,\1)^{op}$ on $H^\b(\C,V)$ via the tensor product.  Namely, we take
\begin{equation}\label{eq:346}
\begin{array}{c}
\left(\xymatrix{
\1 & X\ar[l]_{s}\ar[r]^{f} & \Sigma^{n}\1
}\right)\cdot \left(\xymatrix{
\1 & Y\ar[l]_{t}\ar[r]^{g} & \Sigma^{m}V
}\right)\\\\
=\ \left(\xymatrix{
\1 & X\ot Y\ar[l]_{s\ot t}\ar[r]^{f\ot g} & \Sigma^{n+m}V
}\right).
\end{array}
\end{equation}
One can check that this action agrees with the usual action $H^\b(\msc{C},V)\ot H^\b(\msc{C},\1)\to H^\b(\msc{C},V)$ given by composing morphisms in $\mrm{D}^b(\msc{C})$.
\par

If we consider the object $W\ot V^\ast$, we have the adjunction
\[
H^\b(\C,W\ot V^\ast)=\Ext^\b_\C(V,W)
\]
which explicitly sends a map $fs^{-1}:k\leftarrow X\to \Sigma^nW\ot V^\ast$ to $V\overset{s\ot id}\longleftarrow X\ot V\overset{f'}\to \Sigma^n W$ where $f'$ is $f\ot id_V$ composed with the evaluation $id_{\Sigma^nW}\ot ev_V$ (see~\cite[Proposition\ 2.10.8]{egno15}).

\begin{lemma}\label{lem:385}
The adjunction $H^\b(\C,W\ot V^\ast)=\Ext^\b_\C(V,W)$ is an identification of $H^\b(\C,\1)^{op}$-modules, where $H^\b(\C,\1)^{op}$ acts on $\Ext^\b_\C(V,W)$ via the above algebra map $-\ot V$ to $\Ext^\b_\msc{C}(V,V)$.
\end{lemma}

\begin{proof}
For the class of a homogenous map $fs^{-1}$ in $H^\b(\C,\1)$, and $gt^{-1}\in H^\b(\C,W\ot V^\ast)$ with corresponding $g't^{-1}\in \Ext^\b_\C(V,W)$, one simply needs to compare $fs^{-1}\cdot gt^{-1}$ to $fs^{-1}\cdot g'(t\ot id_V)^{-1}$ under the adjunction.  We have directly
\[
\begin{array}{ll}
fs^{-1}\cdot gt^{-1}=(f\ot g)(s\ot t)^{-1}\\
\hspace{1cm}\longmapsto{$\mrm{adj}$}\ (id\ot ev_V)(f\ot g\ot id_V)(s\ot t\ot id_V)^{-1}=(f\ot g')(s\ot t\ot id_V)^{-1}.
\end{array}
\]
By a direct comparison, as in~\eqref{eq:327}, one see that this last element is $g'(t\ot id_V)^{-1}\circ ((fs^{-1})\ot id_V)$, as desired.
\end{proof}

\begin{lemma}\label{lem:Ext_bifun}
The association
\[
\Ext^\b_\msc{C}:\msc{C}^{op}\times \msc{C}\to H^\b(\msc{C},\1)\text{-}\mrm{mod},\ \ (W,V)\mapsto \Ext^\b_\msc{C}(W,V),
\]
is functorial in both $V$ and $W$.
\end{lemma}

\begin{proof}
The assignment $(W,V)\mapsto H^\b(\msc{C},V\ot W^\ast)$ is clearly a bifunctor.  So the result follows from the natural isomorphism $\Ext^\b_\msc{C}(W,V)\cong H^\b(\msc{C},V\ot W^\ast)$ and Lemma~\ref{lem:385}.
\end{proof}

\begin{remark}\label{rem:global}
One can deduce from Lemma~\ref{lem:385} that for any finite type tensor category $\msc{C}$ we have the equivalent global definition for the Krull dimension
\[
\Kdim\msc{C}=\max\left\{\GKdim\Ext^\bullet_\msc{C}(V,V):V\text{ in }\mrm{Ob}\msc{C}\right\},
\]
where $\GKdim$ denotes the Gelfand-Kirillov dimension.  Similarly, we have the global definition of the finite type property as the condition that each $\Ext^\b_\msc{C}(V,V)$, for $V$ in $\msc{C}$, is finitely generated and finite over its center.
\end{remark}

In the sections that follow we simply refer to $H^\b(\C,\1)$ as acting on the right of $H^\b(\C,V)$, rather than using the left $H^\b(\C,\1)^{op}$-action.

\section{Krull dimension bounds via surjectivity}
\label{sect:HC}

We show that if $F:\D\to\C$ is a surjective tensor functor, and $\D$ is of finite type, then $\C$ is also of finite type.  We also bound the Krull dimension of $\C$ by the Krull dimension of $\D$ in this case. Specific applications are given in Section~\ref{sect:appI}, and explicit explanations are given in Hopf theoretic terms throughout.

\subsection{Surjective tensor functors}

Following the conventions of~\cite{etingofostrik04}, a tensor functor $F:\D\to \C$ is exact by definition.  We are free to assume that any tensor functor $F$ is such that $F(\1_\msc{D})=\1_\msc{C}$, and make this assumption here (see~\cite[Remark\ 2.4.6]{egno15}).  We recall the following basic definition.

\begin{definition}
A tensor functor $F:\msc{D}\to \msc{C}$ is called surjective if every object of $\C$ is a subquotient of $F(X)$ for some $X$ in $\C$.
\end{definition}

Restriction $\mrm{res}_f:\rep(B)\to \rep(A)$ along a Hopf map $f:A\to B$, for example, is surjective if and only if $f$ is injective.  Dually, if we employ corepresentations, restriction $\mrm{res}^w:\mrm{corep}(\Sigma)\to \mrm{corep}(\Lambda)$ along a Hopf map $w:\Sigma\to \Lambda$ is surjective if and only if $w$ is surjective.  To complete the analogy, we also expect a surjective tensor functor to be faithful.  However, faithfulness holds in general, for tensor functors between (finite, non-multi) tensor categories.

\begin{lemma}[{\cite{etingofostrik04}}]\label{lemma:surj_func_faithful}
\begin{enumerate}
\item[(i)] For (finite) tensor categories $\msc{D}$ and $\msc{C}$, any tensor functor $F:\msc{D}\to \msc{C}$ is faithful.
\item[(ii)] If $F$ is surjective, then $F$ also admits an exact right adjoint $I:\msc{C}\to \msc{D}$.
\item[(iii)] If $F$ is surjective, then $F$ preserves projectives.
\end{enumerate}
\end{lemma}

\begin{proof}
(i) Since $F$ is exact by definition, it suffices to show that the only object mapped to $0$ under $F$ is the zero object in $\msc{D}$.  But this simply follows by preservation of Frobenius-Perron dimension, and the fact that an object vanishes if and only if it is of $0$ Frobenius-Perron dimension.  (ii) Existence and exactness of the right adjoint follows by~\cite[Corollary 3.15]{etingofostrik04}.  (iii)~\cite[Corollary 3.22]{etingofostrik04}.
\end{proof}

In the case of a Hopf inclusion $f:A\to B$, the right adjoint is given by the cotensor $(B^\ast\ot -)^{A^\ast}:\rep A\to \rep B$, as explained in~\cite[Proposition 6]{doi81}.  Exactness follows by coflatness for finite dimensional Hopf quotients~\cite{nicholszoeller89}.

\subsection{Surjectivity and the finite type property}

Consider $F:\msc{D}\to \msc{C}$ a surjective tensor functor.  For the units we have $\1_\C=F(\1_\D)$ and the $(F,I)$-adjunction provides a natural identification
\[
\Hom_\msc{D}(\1,I(V))\overset{\cong}\to \Hom_\msc{C}(\1,V),
\]
which is given by applying $F$ and the counit $\epsilon^V:FI(V)\to V$, for arbitrary $V$ in $\msc{C}$.  Since $F$ is exact and preserves projectives (or since $I$ is exact and preserves injectives), this natural identification extends to an identification of graded extensions
\begin{equation}\label{eq:620}
H^\b(\msc{D},I(V))\overset{\cong}\to H^\b(\msc{C},V).
\end{equation}
We consider $H^\b(\msc{C},V)$ as a $H^\b(\msc{D},\1)$-module via the algebra map $F:H^\b(\msc{D},\1)\to H^\b(\msc{C},\1)$.

\begin{proposition}\label{prop:fgcat}
If $F:\D\to \C$ is a surjective tensor functor, and $\D$ is of finite type, then $\C$ is also of finite type.  Furthermore, in this case
\begin{enumerate}
\item[(i)] the induced map $H^\b(\msc{D},\1)\to H^\b(\msc{C},\1)$ is a finite algebra map and
\item[(ii)] the Krull dimension of $\msc{C}$ is bounded as $\Kdim \C\leq \Kdim \D$.
\end{enumerate}
\end{proposition}

\begin{proof}
For arbitrary $V$ in $\msc{C}$, the adjunction~\eqref{eq:620} is an isomorphism of $H^\b(\D,\1)$-modules, as it is a composition of the module maps $F:H^\b(\msc{D},I(V))\to H^\b(\msc{C},FI(V))$ and $\epsilon^V_\ast:H^\b(\msc{C},FI(V))\to H^\b(\msc{C},V)$.  Hence each $H^\b(\msc{C},V)$ is finite over $H^\b(\msc{D},\1)$.  Since $H^\b(\msc{D},\1)$ acts through $H^\b(\msc{C},\1)$ this implies that $\msc{C}$ is of finite type.  One obtains (i) by considering the case $V=\1$ and (ii) follows from (i).
\end{proof}

Applying this to the Hopf case yields

\begin{corollary}\label{cor:fgsubalg}
Suppose we have a Hopf inclusion $A\to B$, and that $B$ has finitely generated cohomology.  Then $A$ has finitely generated cohomology as well, and the Krull dimension for $A$ is bounded by that of $B$, $\Kdim H^\b(A,k)\leq \Kdim H^\b(B,k)$.
\end{corollary}

Considering the case in which $B$ is semisimple, i.e.\ of finite type and of Krull dimension 0, reproduces the familiar result that Hopf subalgebras of semisimple Hopf algebras are semisimple (see e.g.~\cite{montgomery93}).

\subsection{A remark on module categories}
\label{sect:module}

Recall that a module category over a (finite) tensor category $\msc{C}$ is a (finite) $k$-linear category $\msc{M}$ equipped with an action $\ot:\msc{C}\times \msc{M}\to \msc{M}$ and associativity constraint which satisfies all expected axioms~\cite{ostrik03}.  We call $\msc{M}$ {\it exact}~\cite{etingofostrik04} if for any projective $P$ in $\msc{C}$, and arbitrary $V$ in $\msc{M}$, the object $P\ot V$ is projective in $\msc{M}$.  We remark on the Hopf case following Proposition~\ref{prop:M} below.
\par

One can similarly apply~\cite[Corollary 3.15]{etingofostrik04} to deduce a finite generation property for exact module categories over finite type categories.  In considering module categories, we employ the global definitions of the finite type property and Krull dimension introduced in Remark~\ref{rem:global}.

\begin{proposition}\label{prop:M}
If $\msc{C}$ is of finite type, and $\msc{M}$ is an exact module category over $\msc{C}$, then $\msc{M}$ is also of finite type and of bounded Krull dimension $\Kdim\msc{M}\leq \Kdim\msc{C}$.
\end{proposition}

\begin{proof}[Sketch proof]
One argues directly, via a formula as in~\eqref{eq:346}, to find that for each $Y$ in $\msc{M}$ the map $-\ot Y:H^\ast(\msc{C},\1)\to \Ext_\msc{M}(Y,Y)$ has central image.  Since the functor $-\ot Y:\msc{C}\to \msc{M}$ preserves projectives and admits an exact right adjoint~\cite[Corollaries 3.15, 3.22; Lemma 3.21]{etingofostrik04} one can argue as above to obtain the proposed result.
\end{proof}

The proposition obviously implies a broader version of the finite generation conjecture for both tensor categories and module categories.  In Hopf language, exact module categories over some $\rep(A)$ correspond to simple $A$-comodule algebras~\cite{andruskiewitschmombelli07}.  In particular, representation $\rep(T)$ for comodule subalgebras $T\subset A$ provide exact module categories over $\rep(A)$, and such $\rep(T)$ are therefore of finite type whenever $\rep(A)$ is of finite type.

\section{The Drinfeld center, (de-)equivariantization, and generic examples}
\label{sect:appI}

We apply Proposition~\ref{prop:fgcat} to relate the cohomology of dual categories $\msc{C}^\ast_\msc{M}$ to the cohomology of the Drinfeld center $Z(\msc{C})$.  In the Hopf setting, the Drinfeld center is the representation category of the Drinfeld double $Z(\rep A)\cong \rep D(A)$~\cite[Proposition 7.14.6]{egno15} and, as explained in the introduction, the dual categories $(\rep A)^\ast_\msc{M}$ include representation categories of the linear dual $A^\ast$ and arbitrary cocycle twists $A_\sigma$.

We discuss behavior of cohomology under (de-)equivariantization, and give a number of new examples of finite type tensor categories in finite characteristic.  In Section~\ref{sect:braidedII}, we return to the subject of (de-)equivariantization and give a much stronger result regarding the behavior of cohomology under the combined processes of (de-)equivariantization and taking the Drinfeld center.

\subsection{Duality via exact module categories}
\label{sect:intro_duals}

Following~\cite{ostrik03,etingofostrik04}, we define for an exact module category $\msc{M}$ over $\msc{C}$ the dual $\msc{C}^\ast_\msc{M}$ as the endomorphism category
\[
\msc{C}^\ast_\msc{M}:=\End_{\msc{C}\text{-mod}}(\msc{M},\msc{M}),
\]
i.e.\ the category of right exact, $k$-linear, $\msc{C}$-module endomorphisms of $\msc{M}$.  This category is $k$-linear and monoidal under composition of functors.  Furthermore, exactness of $\msc{M}$ implies that $\msc{C}^\ast_\msc{M}$ is finite and rigid~\cite{etingofostrik04}, and therefore a tensor category.
\par

The relation $\msc{C}\sim_\mrm{Mor}\msc{D}\ \Leftrightarrow\ \exists$ exact $\msc{M}$ for which $\msc{D}^\mrm{cop}\cong\msc{C}^\ast_\msc{M}$ is known to be an equivalence relation (Morita equivalence)~\cite[Proposition 7.12.18]{egno15}.  Also, $\msc{C}^\ast_\msc{M}$ is of the same (Frobenius-Perron) dimension as $\msc{C}$, and there is an obvious (exact) action of $\msc{C}^\ast_\msc{M}$ on $\msc{M}$ so that $(\msc{C}^\ast_\msc{M})^\ast_\msc{M}\cong \msc{C}$~\cite{etingofostrik04}. In the case of a Hopf algebra $\msc{C}=\rep(A)$ particular choices of $\msc{M}$ recover familiar examples of linear duality and cocycle twists, as explained in the introduction (see e.g.~\cite[Ex 7.12.27]{egno15} and Lemma~\ref{lem:1025}).  

\subsection{The Drinfeld center and stability of cohomology under duality}

We apply Proposition~\ref{prop:fgcat} to the forgetful functor $F:Z(\C)\to \C$ to obtain the following result.

\begin{corollary}\label{cor:Z2C}
If the Drinfeld center $Z(\C)$ of $\msc{C}$ is of finite type then $\msc{C}$ is of finite type, as is any dual $\C^\ast_\msc{M}$ of $\C$ with respect to an exact module category $\msc{M}$. Furthermore, the Krull dimensions are uniformly bounded
\[
\Kdim \C^\ast_\msc{M}\leq \Kdim Z(\C).
\]
\end{corollary}

In a more concise language, the corollary says that the entire categorical Morita equivalence class of $\C$ is of finite type whenever the center $Z(\C)$ is of finite type.

\begin{proof}
For any exact $\C$-module category $\M$, we have an equivalence of braided tensor categories $Z(\C^\ast_\M)\cong Z(\C)^\mrm{rev}$~\cite[Corollary 3.35]{etingofostrik04}.  Hence there is a surjective tensor functor $F:Z(\C)\to \msc{C}^\ast_\msc{M}$ by~\cite[Proposition\ 3.39]{etingofostrik04}.  So we apply Proposition~\ref{prop:fgcat} to see that $\msc{C}^\ast_\msc{M}$ is of finite type and with bounded Krull dimension as proposed.
\end{proof}

For Hopf algebras we have

\begin{corollary}\label{cor:DA}
If the Drinfeld double $D(A)$ of a Hopf algebra $A$ has finitely generated cohomology then
\begin{itemize}
\item any cocycle twist $A_\sigma$ of $A$ has finitely generated cohomology;
\item the dual Hopf algebra $A^\ast$ has finitely generated cohomology;
\item any cocycle twist of the dual $(A^\ast)_J$ has finitely generated cohomology;
\end{itemize}
and the Krull dimensions are uniformly bounded
\[
\Kdim H^\b(A_\sigma,k),\ \Kdim H^\b(A^\ast,k),\ \Kdim H^\b((A^\ast)_J,k)\ \leq\!\ \Kdim H^\b(D(A),k).
\]
\end{corollary}

We note that in this case the categorical equivalence between $D(A)$ and $D(A_\sigma)$ can be realized rather concretely, as in~\cite{majidoeckl99,benkartetal10}.  We apply the above corollaries to non-degenerate tensor categories (resp.\ factorizable Hopf algebras) in Section~\ref{sect:nondeg}.

\subsection{Equivariantization and de-equivariantization}

For this subsection we assume $\mathrm{char}(k)= 0$.  This is to avoid a situation in which $\mrm{char}(k)\mid |G|$, for a given finite group $G$.
\par

Let $\msc{E}$ be a braided tensor category.  A {\it central embedding} of $\msc{E}$ into a tensor category $\D$ is defined as a fully faithful tensor functor $Q:\msc{E}\to \D$ along with a choice of braided lifting $\tilde{Q}:\msc{E}\to Z(\D)$.
\par

Recall that for any central embedding $\rep(G)\to \D$ we can define the de-equivariantization $\D_G$, which is the tensor category of $\O(G)$-modules in $\D$, where $\O(G)$ is the linear dual of the group algebra.  As an inverse operation to de-equivariantization, we can form the equivariantization $\C^G$ of any tensor category $\C$ which is equipped with an action of a finite group $G$~\cite{dgno10}.  More precisely, we consider $\rho:G\to \underline{\Aut}_{\ot}(\C)$ a group action on $\C$, where $\underline{\Aut}_{\ot}(\C)$ denotes the $2$-group of tensor autoequivalences of $\C$.  Then the equivariantization $\C^G$ consists of objects $V$ in $\C$ equipped with compatible isomorphisms $g^V:\rho(g)V\to V$, for all $g\in G$.  Morphisms $l:V\to W$ between objects in $\C^G$ are exactly those maps in $\C$ which commute with the isomorphisms $g^?$.

\begin{theorem}\label{thm:HGC}
Suppose $F:\D\to \C$ is a de-equivariantization of $\D$ with respect to a central embedding $\rep(G)\to \D$, or equivalently a equivariantization of $\C$ with respect to a $G$-action.  Then $\D$ is of finite type if and only if $\C$ is of finite type, and in this case the Krull dimensions agree $\Kdim\msc{D}=\Kdim\msc{C}$.
\end{theorem}

Before providing the proof of Theorem~\ref{thm:HGC} we establishing some background materials.  Recall from~\cite{dgno10} that for the equivariantization $F:\C^G\to \C$, the induction $I:\C\to \C^G$ sends each object in $\C$ to its orbit under the $G$-action.  In particular, each $V$ in $\C$ is a summand of $FI(V)$.  It follows that the forgetful functor $F$ is surjective.  By writing instead $\C^G=\D$, $\C=\D_G$, we have that any de-equivariantization $F:\D\to \D_G$ is surjective.  Whence we apply Proposition~\ref{prop:fgcat} to arrive at the following result.

\begin{lemma}\label{lem:HD_G}
Whenever $\D$ is of finite type, the de-equivariantization $\D_G$ is of finite type.
\end{lemma}

Note that for the equivariantization $\msc{C}^G$ there is an identification $\Hom_{\C^G}(V,W)=\Hom_{\C}(V,W)^G$, where $G$ acts by $g\cdot f=g^W(\rho(g)f)(g^V)^{-1}$.  By exactness of the invariants functor $(-)^G$ we have $\Ext^\b_{\C^G}(V,W)=\Ext^\b_\C(V,W)^G$, and in particular $H^\b(\C^G,V)=H^\b(\C,V)^G$.  We now prove the theorem.

\begin{proof}[Proof of Theorem~\ref{thm:HGC}]
One direction follows directly from Lemma~\ref{lem:HD_G}.  That is, $\C$ is of finite type whenever $\D$ is. 

For the converse, suppose $\msc{C}$ is of finite type.  We have $H^\b(\D,V)=H^\b(\C,FV)^G$ for each $V$ in $\msc{D}$.  Since $H^\b(\C,\1)$ is finitely generated, the invariants $H^\b(\C,\1)^G = H^\b(\D,\1)$ are also finitely generated, and $H^\b(\D,\1)\to H^\b(\C,\1)$ is a finite algebra extension~\cite{springer}.  Hence the Krull dimensions of these algebras are equal.  All that is left is to prove that each $H^\b(\D,V)$ is a finitely generated module over $H^\b(\D,\1)$.
\par

Consider any object $V$ in $\msc{D}$.  Since $H^\b(\C,\1)$ is finite over $H^\b(\D,\1)$, the finite type property for $\msc{C}$ implies that each $H^\b(\C,FV)$ is finite over $H^\b(\D,\1)$. Since $H^\b(\D,\1)$ is Noetherian, the submodule $H^\b(\D,V)=H^\b(\C,FV)^G$ is therefore finite over $H^\b(\D,\1)$ as well.  Hence $\msc{D}$ is of finite type.
\end{proof}

\begin{remark}
As the familiar reader may recognize, Theorem~\ref{thm:HGC} is also provable by simply considering the identification $H^\b(\D,V)=H^\b(\C,FV)^G$ and applying basic commutative algebra.
\end{remark}

\subsection{Examples in finite characteristic: Frobenius kernels}
\label{sect:O_sig}

We give here some examples in finite characteristic.  Generic examples in characteristic $0$ are given in Sections~\ref{sect:char0}, and more explicit examples appear in Section~\ref{sect:uqT}.
\par

Let $k$ be a field of (finite) odd characteristic $p$.\footnote{This assumption on the characteristic is not optimal.  See~\cite[Section\ 4.3]{friedlandernegron18}.}  Let $\mbb{G}$ be a smooth algebraic group over $k$, and $\mbb{G}_{(r)}$ denote the $r$-th Frobenius kernel in $\mbb{G}$.  Rather, $\mbb{G}_{(r)}$ is the group scheme theoretic kernel of the $r$-th Frobenius map $\mbb{G}\to \mbb{G}^{(r)}$.  We let $\O=\O(\mbb{G}_{(r)})$ denote the (commutative) algebra of global functions on $\mbb{G}_{(r)}$, and $k\mbb{G}_{(r)}$ denote the vector space dual $k\mbb{G}_{(r)}=\O^\ast$.
\par

In~\cite{friedlandernegron18}, Friedlander and the first author show that the double $D(k\mbb{G}_{(r)})$ has finitely generated cohomology.

\begin{theorem}[{\cite[Theorem\ 5.3]{friedlandernegron18}}]\label{thm:fn}
For an arbitrary smooth algebraic group $\mbb{G}$ in characteristic $p$, the Drinfeld double $D(k\mbb{G}_{(r)})$ has finitely generated cohomology.  Furthermore, we calculate the Krull dimension
\[
\Kdim H^\b(D(k\mbb{G}_{(r)}),k)=\Kdim H^\b(k\mbb{G}_{(r)},k)+\dim\mbb{G}.
\]
\end{theorem}

Here $\dim\mbb{G}$ is the dimension of $\mbb{G}$ as a variety.  As another point of interest, Gelaki has classified exact module categories for arbitrary finite group schemes in terms of cohomological data~\cite{gelaki15}.  We present his result in the particular case of a Frobenius kernel in a smooth algebraic group $\mbb{G}$.

\begin{theorem}[{\cite[Theorem\ 3.9]{gelaki15}}]
Exact module categories over $\mathrm{Coh}(\mbb{G}_{(r)})=\rep(\O)$ are classified by pairs consisting of a choice of closed subgroup $H\subset \mbb{G}_{(r)}$ and $2$-cocycle $\psi:kH\ot kH\to k$.  For any pair $(H,\psi)$ the associated module category $\M(H,\psi)$ is the category of $H$-equivariant sheaves on $\mbb{G}_{(r)}$, under the translation action $\mbb{G}_{(r)}\times H\to \mbb{G}_{(r)}$, with associativity given by $\psi$.
\end{theorem}

In the case $\psi=1$ the module category $\M(H,1)$ is equivalent to the category of coherent sheaves on the quotient $G/H=\opn{Spec}(\O(G)^H)$.  This equivalence is derived from the fact that the quotient $G\to G/H$ is $H$-Galois, so that we may apply descent to find $\mathrm{Coh}(G)^H\cong \mathrm{Coh}(G/H)$.  The module structure on $\mrm{Coh}(G/H)$ is induced by the pushforward functor $\mathrm{Coh}(G)\to \mathrm{Coh}(G/H)$.
\par

One applies Proposition~\ref{prop:fgcat} and Theorem~\ref{thm:fn} to arrive at the following corollary.

\begin{corollary}\label{cor:repOM}
Let $\mbb{G}$ be a smooth algebraic group in characteristic $p$.  For any pair $(H,\psi)$ consisting of a closed subgroup $H$ of $\mbb{G}_{(r)}$, and $2$-cocycle $\psi$ for $H$, the corresponding dual category $\rep(\O)^\ast_{\M(H,\psi)}$ is of finite type.  Furthermore, the Krull dimensions of the duals are uniformly bounded
\[
\Kdim \rep(\O)^\ast_{\M(H,\psi)}\leq \Kdim \rep(\mbb{G}_{(r)})+\dim\mbb{G}.
\]
\end{corollary}

Of course, it is already very interesting to consider cocycle twists of $\O$.  We recall that $\rep D(\O)\cong \rep D(\O_\sigma)$ for any $2$-cocycle $\sigma$~\cite{majidoeckl99,benkartetal10}.

\begin{corollary}\label{cor:O_sig}
For any smooth algebraic group $\mbb{G}$, Frobenius kernel $\mbb{G}_{(r)}$, and $2$-cocycle $\sigma$ for the coordinate algebra $\O=\O(\mbb{G}_{(r)})$, the twisted algebra $\O_\sigma$ has finitely generated cohomology.  Furthermore,
\[
\Kdim \rep(\O_\sigma)\leq \Kdim \rep(\mbb{G}_{(r)})+\dim\mbb{G}.
\]
\end{corollary}

We note that the twisted algebra $\O_\sigma$ is neither commutative nor cocommutative in general.  Gelaki provides in~\cite[Example 6.14]{gelaki15} an explicit example of a cocycle twist $\sigma$ of the algebra of functions for the semi-direct product $(\mbb{G}_a\rtimes \mbb{G}_m)_{(1)}$ for which the corresponding cocycle twisted algebra is neither commutative nor cocommutative.  As an algebra, one has explicitly
\[
\O\left((\mbb{G}_a\rtimes \mbb{G}_m)_{(1)}\right)_{\sigma}=k\langle x, t\rangle/(x^p, t^p-1,[x,t]-t^2+2t-1),
\]
while the coalgebra structure is unchanged, and hence remains non-cocommutative.
\par

This single cocycle proliferates to produce an infinite class of such Hopf algebras.  Indeed, given any embedding $(\mbb{G}_a\rtimes \mbb{G}_m)_{(1)}\to \mbb{G}_{(r)}$, for arbitrary smooth $\mbb{G}$, we can restrict this cocycle to $\mbb{G}_{(r)}$ to arrive at a Hopf algebra in characteristic $p$ which is neither commutative nor cocommutative.  Of course, there should be many more interesting example of cycle twisted algebras which are not coming from the above small example as well.
\par

We can also complete some work from~\cite[Section\ 5.3]{friedlandernegron18} by applying Corollary~\ref{cor:fgsubalg}.  For smooth $\mbb{G}$, any $r>0$, and group scheme quotient $\mbb{G}_{(r)}\to G'$, one can form the relative double $D(G',\mbb{G}_{(r)})$, which as an algebra is the smash product $\O(G')\# k\mbb{G}_{(r)}$ under the adjont action of $\mbb{G}_{(r)}$ on $G'$.  This algebra has representation category equivalent to the relative center of the corresponding module category $\mrm{rep}(G')$ over $\mrm{rep}(\mbb{G}_{(r)})$ (see Section~\ref{sect:braidedI}).  The following result was proved for $G'$ of the form $\mbb{G}_{(r)}/\mbb{G}_{(r-i)}$ in~\cite{friedlandernegron18}.

\begin{corollary}[{cf.~\cite[Theorem\ 5.7]{friedlandernegron18}}]
For any smooth algebraic group $\mbb{G}$, Frobenius kernel $\mbb{G}_{(r)}$, and quotient map $\mbb{G}_{(r)}\to G'$ of group schemes, the relative double $D(G',\mbb{G}_{(r)})$ has finitely generated cohomology.  Additionally, the Krull dimension is bounded
\[
\Kdim H^\b(D(G',\mbb{G}_{(r)}),k)\leq \Kdim \rep(\mbb{G}_{(r)})+\dim\mbb{G}.
\]
\end{corollary}

\begin{proof}
We have the Hopf inclusion $D(G',\mbb{G}_{(r)})\to D(\mbb{G}_{(r)})$ and apply Corollary~\ref{cor:fgsubalg} in conjunction with Theorem~\ref{thm:fn}.
\end{proof}

\begin{remark}
To our knowledge, the cocycle twisted algebras $\O(\mbb{G}_{(r)})_\sigma$, in addition to the doubles of~\cite{friedlandernegron18}, provide the first examples of noncommutative and noncocommutative Hopf algebras in characteristic $p$ for which the cohomology is known to be finitely generated (see also~\cite{nguyenwangwitherspoon19,erdmannsolbergwang}, where the authors consider cohomology with trivial coefficients). 
\end{remark}

\begin{remark}
The material of this subsection focuses on infinitesimal group schemes.  One could, at the other extreme, consider finite (discrete) groups in characteristic $p$.  In this setting, again, module categories have been classified~\cite[Section\ 4]{etingofostrik04}.  Indeed, the work of~\cite{etingofostrik04} precedes and motivates the work of~\cite{gelaki15}.  It is easily checked that the double $D(kG)$ of a finite group $G$ has finitely generated cohomology and $\Kdim \rep(D(kG))=\Kdim \rep(G)$.  So we see that all duals $\rep(G)^\ast_\msc{M}$ are of finite type for an arbitrary finite group $G$.
\end{remark}

\subsection{Examples in characteristic $0$: pointed Hopf algebras}
\label{sect:char0}

We consider the pointed Hopf algebras $u(\mfk{D})=u(\mcl{D},\lambda,\mu)$ of Andruskiewitsch and Schneider~\cite{andruskiewitschschneider10}, which are Hopf algebras over $\mbb{C}$.  We call such $u(\mfk{D})$ Cartan-type pointed Hopf algebras, as they are cocycle deformations of bosonizations of Nichols algebras of Cartan-type braided vector spaces.  These Hopf algebras are generalizations of small quantum groups.  The algebras of~\cite{andruskiewitschschneider10} are ``generic" among all finite dimensional pointed Hopf algebras with abelian group of grouplikes (cf.~\cite{angiono13}).

\begin{theorem}\label{thm:char0}
Consider a Cartan-type pointed Hopf algebra $u(\mfk{D})$.  For any exact $\rep(u(\mfk{D}))$-module category $\msc{M}$, the dual $\rep(u(\mfk{D}))^\ast_\msc{M}$ is of finite type.  Furthermore, if we let $\Phi$ denote the root system associated to the data $\mfk{D}$, we have
\[
\Kdim \mrm{rep}\left(u(\mfk{D})\right)^\ast_\msc{M}\leq 2|\Phi^+|,
\]
for arbitrary exact $\msc{M}$.
\end{theorem}

\begin{proof}
Take $\msc{C}=\rep(u(\mfk{D}))$ and take $G=G(u(\mfk{D}))$.  We will show that $Z(\msc{C})$ is of finite type and of Krull dimension $\leq 2|\Phi^+|$.  We may assume that $u(\mfk{D})$ is coradically graded, as $u(\mfk{D})$ is a cocycle deformation of $\operatorname{gr} u(\mfk{D})$, and is hence in the same Morita equivalence class.  In this case $u(\mfk{D})$ is the bosonization $B(V)\rtimes G$ of the Nichols algebra of Cartan-type braided vector space $V$.  The Drinfeld double $D$ of such coradically graded $u(\mfk{D})$ is then a cocycle deformation of the bosonizations
\[
D=\left(B(V)\rtimes G\ot B(V^\ast)^{cop}\rtimes G^\vee\right)_\sigma,
\]
by~\cite{doitakeuchi94}.  For the cooposite, we invert the $G$-action on $B(V)$, which reverses the comultiplication on the $G^\vee$-bosonization.
\par

We note that $V$ and $V^\ast$ have the same braiding matrix, so that $B(V)\cong B(V^\ast)$ as algebras.  One can see this by decomposing $V$ into simultaneous eigenspaces for the (commuting) actions of $G$ and $G^\vee$.  Also $B(V)^{cop}\cong B(W)$ where $W=V$ with entrywise inverted braiding matrix.  Hence $D$ is a Cartan-type pointed Hopf algebra.  It follows by~\cite[Proof of Theorem 6.3]{mpsw10} that $D$ has finitely generated cohomology of Krull dimension less than or equal to $2|\Phi^+|=\Kdim H^\b(B(V)\ot B(V^\ast),\mbb{C})$.  Hence $Z(\msc{C})$ is of finite type and of Krull dimension $\leq 2|\Phi^+|$, and the result now follows from Corollary~\ref{cor:Z2C}.
\end{proof}

In finite characteristic we considered cocycle deformations of function algebras of infinitesimal group schemes.  For pointed Hopf algebras in characteristic $0$, cocycle deformations of the Cartan-type $u(\mfk{D})$ do not produce new examples of Hopf algebras with finitely generated cohomology, as such algebras are stable under cocycle deformation and hence are covered in~\cite{mpsw10}.  However, we discuss the specific examples of dynamical quantum groups in Section~\ref{sect:uqT}.

\section{Deligne products and non-degenerate categories}
\label{sect:nondeg}

In this section we establish some basic results concerning Deligne products of finite type categories and base change.  We also show that if $\C$ is non-degenerate (braided) and of finite type, then every dual category $\C^\ast_\M$ with respect to an exact module category $\M$ is also of finite type.  This means that the finite type property is preserved by Morita equivalences in the non-degenerate case. At times in this section it will be convenient to employ an equivalence $\msc{C}\simeq \mrm{rep}(A)$ to express an arbitrary tensor category as the category of representations over a Hopf {\it algebroid} $A$.  We recall the notion of a Hopf algebroid in Appendix~\ref{sect:A}, and reconstruct a standard equivalence $\msc{C}\simeq \mrm{rep}(A)$ at Lemma~\ref{lem:CasA} therin.

\subsection{Some standard lemmas}
\label{sect:lems}

Let us collect a few standard results, which seem not to have appeared in an organized manner in the literature.

\begin{lemma}[{cf.~\cite{roquier}}]\label{lem:irreps}
Suppose the cohomology $H^\b(\C,\1)$ is a Noetherian algebra.  Then the following are equivalent:
\begin{enumerate}
\item[(i)] The cohomology $H^\b(\C,V)$ is finitely generated as a $H^\b(\C,\1)$-module for each simple object $V$ in $\C$.
\item[(ii)] The cohomology $H^\b(\C,W)$ is finite generated as a $H^\b(\C,\1)$-module for arbitrary $W$ in $\C$.
\end{enumerate}
\end{lemma}

\begin{proof}
Obviously (ii) implies (i).  Suppose now that (i) holds.  We proceed by induction on the length of $W$.  As our base case, (i) tells us the cohomology $H^\b(\C,W)$ is finitely generated whenever $W$ is length $1$.  Now suppose the cohomology is finitely generated for each object of length $<l$, and consider a length $l$ object $W$.  We place $W$ in an exact sequence $0\to W'\to W\to V\to 0$, with $W'$ and $V$ of length $<l$, and get an exact sequence on cohomology $H^\b(\C,W')\to H^\b(\C,W)\to H^\b(\C,V)$.  (Exactness at the middle term follows from the standard long exact sequence on cohomology.)
\par

Since $H^\b(\C,\1)$ is Noetherian, and $H^\b(\C,V)$ is a finitely generated module, the image $N$ of $H^\b(\C,W)$ in $H^\b(\C,V)$ is finitely generated, as is the image $M$ of $H^\b(\C,W')$ in $H^\b(\C,W)$.  Hence we have an exact sequence $0\to M\to H^\b(\C,W)\to N\to 0$ of $H^\b(\C,\1)$-modules, with $M$ and $N$ finitely generated.  It follows that $H^\b(\C,W)$ is finitely generated.  Now (ii) follows by induction.
\end{proof}

For a Hopf algebroid $A$ over $R$ we can change base along any field extension $k\to k'$ to produce a new weak Hopf algebroid $A_{k'}$ over $R_{k'}$.  For any tensor category $\C$ we define $\C_{k'}$ by choosing an equivalence $\C\cong \rep(A)$ with $A$ a Hopf algebroid and taking $\C_{k'}=\rep(A_{k'})$.  One can check that the base change is unique up to tensor equivalence.  We say a given tensor category over $k$ is {\it defined over} a subfield $k''\subset k$ if it is equivalent to the base change $\msc{C}\cong (\msc{C}'')_k$ of a tensor category $\msc{C}''$ over $k''$.

\begin{lemma}\label{lem:basechange}
Let $k\to k'$ be a field extension and $\msc{C}$ be a tensor category over $k$.  If the base change $\msc{C}_{k'}$ is of finite type over $k'$ then $\C$ is of finite type over $k$.  In this case $\Kdim\msc{C}=\Kdim\msc{C}_{k'}$.  Furthermore, if $\msc{C}$ is defined over a perfect field then the converse implication holds as well.
\end{lemma}

\begin{proof}
Suppose that $\msc{C}_{k'}$ is of finite type.  Since $H^\b(\C_{k'},\1_{k'})$ is finitely generated then we can choose $N>0$ so that $H^\b(\C_{k'},\1_{k'})$ is generated in degrees $\leq N$.  Hence the natural map $\Sym(H^{\leq N}(\C_{k'},\1_{k'}))\to H^\b(\C_{k'},\1_{k'})$ is surjective.  By the diagram
\[
\xymatrix{
k'\ot \Sym(H^{\leq N}(\C,\1))\ar[d]_\cong\ar[r]& k'\ot H^\b(\C,\1)\ar[d]_\cong\\
\Sym(H^{\leq N}(\C_{k'},\1_{k'}))\ar[r]& H^\b(\C_{k'},\1_{k'}),
}
\]
and faithful flatness of $k'$ over $k$, we conclude that the map $\Sym(H^{\leq N}(\C,\1))\to H^\b(\C,\1)$ is surjective.  Since $H^\b(\C,\1)$ is finite dimensional in each degree, it follows that $H^\b(\C,\1)$ is a finitely generated algebra.  A similar argument shows that each $H^\b(\C,V)$ is a finitely generated module over $H^\b(\C,\1)$.
\par

Conversely, suppose $\msc{C}$ is defined over a perfect field and let $(-)_{k'}:\C\to \C_{k'}$ denote the base change functor.  We have an identification of algebras $H^\b(\C_{k'},\1_{k'})=k'\ot H^\b(\C,\1)$, and an identification of $H^\b(\C_{k'},\1_{k'})$-modules $H^\b(\C_{k'},V_{k'})=k'\ot H^\b(\C,V)$ for every object $V$ in $\C$.  Our perfectness assumption implies that each simple object in the base change $\msc{C}_{k'}$ is a summand of some $V_{k'}$, for a simple $V$ in $\msc{C}$.  One can therefore apply Lemma~\ref{lem:irreps} to see that $\msc{C}_{k'}$ is of finite type whenever $\msc{C}$ is of finite type.
\par

As for the Krull dimension, any finite map $k[X_1,\dots,X_n]\to H^\b(\msc{C},\1)$ changes base to a finite map $k'[X_1,\dots,X_n]\to H^\b(\msc{C}_{k'},\1)$.  Whence $\Kdim\msc{C}_{k'}=\Kdim\msc{C}$, by Noether normalization.
\end{proof}

\begin{lemma}\label{lem:tp}
If $\msc{C}$ and $\msc{C}'$ are defined over a perfect field then the following conditions are equivalent:
\begin{enumerate}
\item[(i)] Both $\C$ and $\C'$ are of finite type;
\item[(ii)] The Deligne product $\C\bt\C'$ is of finite type.
\end{enumerate}
When either {\rm (i)} or {\rm (ii)} holds we have $\mrm{Kdim}( \msc{C}\bt\msc{C}')=\Kdim\msc{C}+\Kdim\msc{C}'$.
\end{lemma}

\begin{proof}
We have for the algebraic closure $k\to \bar{k}$ that $(\C\bt\C')_{\bar{k}}\simeq\C_{\bar{k}}\bt\C'_{\bar{k}}$, and hence we may assume $k$ is algebraically closed, by Lemma~\ref{lem:basechange}.  Notice that the second product is ``over $\bar{k}$".  In this case the simples of $\C\bt \C'$ are exactly the products $V\bt V'$ of simples $V$ from $\C$ and $V'$ from $\C'$.  We have a canonical isomorphism of algebras
\begin{equation}\label{eq:386}
H^\b(\C,\1)\ot H^\b(\C',\1)\overset{\cong}\to H^\b(\C\bt\C',\1),\ \ [f]\ot[g]\mapsto [f\ot g].
\end{equation}
Similarly, for simple objects $V$ and $V'$ in $\C$ and $\C'$ there is an isomorphism of $H^\b(\C\bt\C',\1)$-modules
\begin{equation}\label{eq:390}
H^\b(\C,V)\ot H^\b(\C',V')\overset{\cong}\to H^\b(\C\bt\C',V\bt V'),
\end{equation}
where $H^\b(\C\bt \C',\1)$ acts on $H^\b(\C,V)\ot H^\b(\C',V')$ via the algebra identification~\eqref{eq:386}.
\par

The fact that these maps are isomorphisms essentially follows from the fact that the product of projective resolutions $P\to\mathbf{1}_\C$ and $P'\to \mathbf{1}_{\C'}$ produces a projective resolution $P\bt P'\to \mathbf{1}_{\C\bt\C'}$.
\par

By the above information, and Lemma~\ref{lem:irreps}, we see that (i) implies (ii).  Suppose now that (ii) holds.  The identification~\eqref{eq:386} implies a surjective algebra map $H^\b(\C\bt\C',\1)\to H^\b(\C,\1)$, and hence that $H^\b(\C,\1)$ is a finitely generated algebra.  From~\eqref{eq:386} and~\eqref{eq:390} we see also that for any object $V$ in $\C$, we have a surjection $H^\b(\C\bt \C',V\bt\mathbf{1}_{\C'})\to H^\b(\C,V)$ of $H^\b(\C\bt\C',\1)$-modules, where we let $H^\b(\C\bt\C',\1)$ act on the codomain via the projection to $H^\b(\msc{C},\1)$.  Since $H^\b(\C\bt \C',V\bt\mathbf{1}_{\C'})$ is finitely generated over $H^\b(\C\bt\C',\1)$, we conclude that $H^\b(\C\bt\C',V)$ is finitely generated over $H^\b(\C\bt\C',\1)$ as well, and hence over $H^\b(\C,\1)$.  Hence $\msc{C}$ is of finite type, and the analogous argument shows $\msc{C}'$ is of finite type as well.
\par

The Krull dimension calculation follows by Noether normalization.
\end{proof}

\subsection{Results for non-degenerate categories}

Recall that a braided tensor category $\C$ is called non-degenerate if any object $V$ for which the square of the braiding $c_{-,V}c_{V,-}:V\ot-\to V\ot -$ is the identity is trivial, that is $V\cong \1^{\oplus n}$.  Non-degeneracy is equivalent to the condition that the canonical braided tensor functor
\[
\C\bt\C^{\mrm{rev}}\to Z(\C),\ \ X\bt \1\mapsto (X,c_{X,-}),\ \1\bt X\mapsto (X,c_{-,X}^{-1})
\]
is an equivalence~\cite{egno15}.  For the category $\rep(A)$ of representations for a quasitriangular Hopf algebra $A$, we have that $\rep(A)$ is non-degenerate if and only if $A$ is factorizable (see e.g.~\cite[Theorem 13.6.1]{radford11}).

\begin{proposition}\label{prop:Cduals}
Suppose $\C$ is non-degenerate, of finite type, and defined over a perfect field.  Then the center $Z(\msc{C})$ is of finite type and of Krull dimension $2\Kdim\msc{C}$.  Furthermore, in this case any dual $\C^\ast_\msc{M}$ with respect to an exact module category $\msc{M}$ is of finite type, and the Krull dimensions are uniformly bounded $\Kdim \msc{C}^\ast_\msc{M}\leq 2\Kdim\msc{C}$.
\end{proposition}

\begin{proof}
Since $\C$ is non-degenerate, we have a tensor equivalence $Z(\C)\cong \C\bt\C^{\mrm{rev}}$. Whence we conclude that $Z(\C)$ has finitely generated cohomology and is of Krull dimension $2\Kdim\msc{C}$, by Lemma~\ref{lem:tp}.  The claim about duals now follows by Corollary~\ref{cor:Z2C}.
\end{proof}

Similarly, if $A$ is a factorizable Hopf algebra with finitely generated cohomology, then all of the conclusions of Corollary~\ref{cor:DA} hold, with the appropriate new bound of Krull dimensions $\leq 2\Kdim H^\b(A,k)$.

\subsection{An equivalence of two conjectures}

One sees from Lemma~\ref{lem:tp} that the stability conjecture has an equivalent expression via the Drinfeld center.

\begin{proposition}\label{prop:conjs}
For tensor categories over a perfect field the following two conjectures are equivalent:
\begin{enumerate}
\item[(A)] Suppose that $\C$ is a tensor category of finite type, and that $\msc{M}$ is an exact $\msc{C}$-module category.  Then the dual $\C^\ast_\M$ is also of finite type and $\Kdim\msc{C}^\ast_\msc{M}=\Kdim\msc{C}$.
\item[(B)] Suppose that $\C$ is a tensor category of finite type.  Then the Drinfeld center $Z(\C)$ is also of finite type and $\Kdim Z(\msc{C})=2\Kdim\msc{C}$.
\end{enumerate}
\end{proposition}

\begin{proof}
Suppose Conjecture A is true and that $\C$ is a finite type tensor category.  Then, by Lemma~\ref{lem:tp}, the product $\C\bt\C^\mrm{cop}$ is of finite type.  We now apply A, in conjunction with Lemma~\ref{lem:tp}, to the tensor equivalence $(\C\bt\C^\mrm{cop})^\ast_\C\cong Z(\C)$~\cite[Proposition 7.13.8]{egno15} to find that $Z(\C)$ is of finite type and that $\Kdim Z(\msc{C})=2\Kdim\msc{C}$.  Hence Conjecture B holds.
\par

Suppose now that Conjecture B holds, i.e.\ that the finite type property is stable under taking centers.  Suppose $\msc{C}$ is of finite type and that $\msc{M}$ is an exact module category over $\msc{C}$.  By B and Corollary~\ref{cor:Z2C} the dual $\msc{C}^\ast_\msc{M}$ is of finite type, and we apply B again to find
\[
\Kdim\msc{C}^\ast_\msc{M}=\frac{1}{2}\Kdim Z(\msc{C}^\ast_\msc{M})=\frac{1}{2}\Kdim Z(\msc{C})=\Kdim\msc{C}.
\]
Hence we establish Conjecture A.
\end{proof}

A similar argument can be applied to obtain a variant of the above proposition for the weak stability conjecture.

\begin{proposition}\label{prop:wconjs}
For tensor categories over a perfect field the following two conjectures are equivalent:
\begin{enumerate}
\item[(wA)] Suppose that $\C$ is a tensor category of finite type, and that $\msc{M}$ is an exact $\msc{C}$-module category.  Then the dual $\C^\ast_\M$ is also of finite type.  Furthermore, there is a polynomial $P\in \mbb{R}_{\geq 0}[X]$ such that $\Kdim\msc{C}^\ast_\msc{M}\leq P(\Kdim\msc{C})$.
\item[(wB)] Suppose that $\C$ is a tensor category of finite type.  Then the Drinfeld center $Z(\C)$ is also of finite type.  Furthermore, there is a polynomial $Q\in \mbb{R}_{\geq 0}[X]$ such that $\Kdim Z(\msc{C})\leq Q(\Kdim\msc{C})$.
\end{enumerate}
\end{proposition}

In the above statements the polynomials $P,\ Q\in\mbb{R}_{\geq 0}[X]$ are to be independent of choice of $\msc{C}$ and $\msc{M}$.  As mentioned in the introduction, Conjecture wB says that the Krull dimension of the center $Z(\msc{C})$ grows sub-exponentially as a function of $\Kdim\msc{C}$.

\section{Dynamical quantum groups}
\label{sect:uqT}

In this section we give some specific examples in characteristic $0$, following the generic conclusions of Theorem~\ref{thm:char0}.  We consider dynamical twists of small quantum groups.  General information on dynamical twists can be found in~\cite{etingofnikshych01,mombelli07}, as well as Appendix~\ref{sect:A} here.  We employ relations between dynamical twists, module categories, and weak Hopf algebras which we outlined in detail in the Appendix.  These relations should be known to experts, but seemingly have not appeared explicitly in the literature to this point.
\par

We fix $\mfk{g}$ a simply-laced, simple, Lie algebra over $\mbb{C}$, i.e.\ $\g$ of type $A$, $D$, $E$, and fix $\Gamma$ a set of simple roots for $\mfk{g}$.  By an abuse of notation, we identify $\Gamma$ with the Dynkin diagram for $\mfk{g}$.
\par

We suppose that $q$ is a root of unity of sufficiently large prime order $l=\operatorname{ord}(q)$.  Specifically, the order of $q$ should be coprime to a finite collection of numbers which depends on the Dynkin type for $\mfk{g}$, in accordance with~\cite[Section\ 5.2]{etingofnikshych01}.  We consider the small quantum group $u_q(\mfk{g})$ as defined in~\cite{lusztig90,lusztig90II}.

\subsection{Small quantum groups and cohomology}

Recall that $u_q(\g)$ is a finite dimensional, factorizable, pointed Hopf algebra.  The algebra $u_q(\g)$ is generated by grouplikes $K_\alpha$ and skew primitives $E_\alpha$ and $F_\alpha$, for $\alpha\in \Gamma$.  The group of grouplikes $G=G(u_q(\g))$ is abelian, and there is a group isomorphism $(\mathbb{Z}/l\mathbb{Z})^{\Gamma}\to G$, $(n_\alpha)_\alpha\mapsto \prod_\alpha K_\alpha^{n_\alpha}$.  We refer the reader to Lusztig's original papers~\cite{lusztig90,lusztig90II} for more information.
\par

By results of Ginzburg-Kumar, the small quantum group $u_q(\g)$ has finitely generated cohomology.  The results of Ginzburg-Kumar were followed by the more general results of~\cite{mpsw10}, which we employed previously in the proof of Theorem~\ref{thm:char0}.

\begin{theorem}[\cite{ginzburgkumar93}, cf.~\cite{bnpp14}]
The cohomology $H^\b(u_q(\g),\mbb{C})$ is a finitely generated algebra, and for any $u_q(\mfk{g})$-representation $V$ the cohomology $H^\b(u_q(\g),V)$ is a finitely generated module over $H^\b(u_q(\g),\mbb{C})$.  Furthermore, there is an algebra isomorphism 
\[
H^\b(u_q(\g),\mbb{C})\cong\O(\mrm{Nil}(\mfk{g})),
\]
where $\mrm{Nil}(\mfk{g})$ is the cone of ad-nilpotent elements in $\mfk{g}$.
\end{theorem}

We note that the above theorem does not require $\mfk{g}$ to be simply-laced.  The remainder of the section is dedicated to an explanation of how one can apply Theorem~\ref{thm:char0}, or alternatively Proposition~\ref{prop:Cduals}, to find that Etingof-Varchenko style dynamical quantum groups~\cite{etingofnikshych01} (cf.~\cite{etingofvarchenko99}) have finitely generated cohomology.

\subsection{A review of dynamical quantum groups and Belavin-Drinfeld triples~\cite{etingofnikshych01}}

\begin{definition}[\cite{belavindrinfeld98}]
A Belavin-Drinfeld triple for $\g$ is a triple $(\Gamma_1,\Gamma_2,T)$, where $\Gamma_1,\Gamma_2\subset \Gamma$ are subgraphs and $T:\Gamma_1\to \Gamma_2$ is a choice of graph isomorphism.  We say a triple $(\Gamma_1,\Gamma_2,T)$ is {\it nilpotent} if for each $\alpha\in\Gamma_1$ there is a positive power $N$ of $T$ so that $T^N(\alpha)\in \Gamma\setminus \Gamma_1$.
\end{definition}

Consider an arbitrary triple $T=(\Gamma_1,\Gamma_2,T)$ for $\g$, i.e.\ one which is not necessarily nilpotent.  One associates to $T$ a dynamical twist
\[
J_{T}:G^\vee\to u_q(\mfk{g})\ot u_q(\mfk{g}),
\]
which is built out of the $R$-matrix for $u_q(\mfk{g})$ and an extension of $T$ to a Hopf map between the quantum Borels $T_+:u_q(\mfk{b}_+)\to u_q(\mfk{b}_-)$, with $T_+(E_\alpha)=E_{T(\alpha)}$ when $\alpha\in \Gamma_1$ and $T_+(E_\alpha)=0$ otherwise.  We refer the reader to~\cite{etingofnikshych01} for the precise construction, and to~\cite{uqJ} for an explicit presentation when $T$ is nilpotent (see also~\cite{ESS}).
\par

As with any dynamical twist, $J_T$ induces a non-trivial associator for the natural $\mrm{rep}(u_q(\mfk{g}))$-action on $\mrm{rep}(G)$ and we obtain new exact, indecomposable $\mrm{rep}(u_q(\mfk{g}))$-module category structure $\msc{M}_T$ on $\mrm{rep}(G)$ (see Section~\ref{sect:dynamM}).  We take the dual to arrive at a tensor category
\[
\mrm{rep}(u_q(\mfk{g}))^\ast_{\msc{M}_T}.
\]
One can explicitly construct from $\msc{M}_T$ a weak Hopf algebra $u_q(\mfk{g},T)$ which is equipped with a canonical equivalence $\mrm{rep}(u_q(\mfk{g},T))\simeq \mrm{rep}(u_q(\mfk{g}))^\ast_{\msc{M}_T}$ (see Lemma~\ref{lem:1025}).  The weak Hopf algebra $u_q(\mfk{g},T)$ is called the (Etingof-Varchenko style) dynamical quantum group associated to the triple $T$~\cite{etingofvarchenko99,etingofnikshych01}.\footnote{Our weak Hopf algebra $u_q(\mfk{g},T)$ is isomorphic to the weak Hopf algebra $D_{\msc{J}_T}$ of~\cite{etingofnikshych01}.  Its dual is the weak Hopf algebra $u_q(\mfk{g})^{J_T}$ of Appendix~\ref{sect:A}}

\begin{remark}
In general, the weak Hopf algebra $u_q(\mfk{g},T)$ realizes solutions to the quantum dynamical, i.e.\ parameter dependent, Yang-Baxter equation.  Furthermore, the dynamical $R$-matrix for the analogously defined dynamical Drinfeld-Jimbo quantum group $U_\hbar(\mfk{g},T)$ (or rather its dual) was employed to provide explicit quantizations of classical $r$-matrices for Lie algebras produced by Belavin and Drinfeld~\cite{ESS,belavindrinfeld98}.  One can see~\cite{etingof02} for a survey on the topic.
\end{remark}

The structure of the category $\mrm{rep}(u_q(\mfk{g}))^\ast_{\msc{M}_T}$ is not well understood at general $T$.  However, for nilpotent $T$ it is clear that the (quantum) parabolics $u_q(\mfk{p}_i)\subset u_q(\mfk{g})$ associated to the subgraphs $\Gamma_i$ appearing in the given triple play a significant role in determining the structure of the dual~\cite{uqJ}.  In the extreme cases of an empty triple $T$ (i.e.\ when $\Gamma_1=\Gamma_2=\emptyset$) or when $T$ an automorphism of the Dynkin diagram the category $\mrm{rep}(u_q(\mfk{g}))^\ast_{\msc{M}_T}$ ``degenerates" to standard quantum groups, and we can understand the situation clearly.

\begin{theorem}[{\cite[Theorem\ 5.4.1]{etingofnikshych01}}]
Let $T$ be a Belavin-Drinfeld triple for $\mfk{g}$.
\begin{enumerate}
\item[(i)] When $T$ is empty, there is an equivalence $\mrm{rep}(u_q(\mfk{g}))^\ast_{\msc{M}_T}\overset{\sim}\to \mrm{rep}(u_q(\mfk{g})^\ast)^{G\times G^{op}}$ of $k$-linear categories, where $G$ acts on $\mrm{rep}(u_q(\mfk{g})^\ast)$ via left and right translation.
\item[(ii)] When $T$ is an automorphism of the Dynkin diagram there is an equivalence of tensor categories $\mrm{rep}(u_q(\mfk{g}))^\ast_{\msc{M}_T}\overset{\sim}\to \mrm{rep}(u_q(\mfk{g}))$.
\end{enumerate}
\end{theorem}

Note the presence of the linear dual in statement (i).

\begin{proof}
(i) The dynamical twist in this case is constant of value $1$, and hence the dynamical quantum group category is equivalent to the proposed equivariantization of $\mrm{rep}(u_q(\mfk{g})^\ast)$, by Proposition~\ref{prop:const}.  (ii) By~\cite[Theorem\ 5.4.1]{etingofnikshych01}, in the case in which $T$ is an automorphism the dynamical quantum group $u_q(\mfk{g},T)^\ast$ is isomorphic to the Xu-style dynamical twisted algebra $u_q(\mfk{g})^{J_T}$ (see Appendix~\ref{sect:dynamM}).  We have then a canonical equivalence of tensor categories
\[
\mrm{Sz}(J_T):\mrm{rep}(u_q(\mfk{g}))\overset{\sim}\to \mrm{rep}(u_q(\mfk{g})^{J_T})
\]
(see again Appedix~\ref{sect:dynamM}).
\end{proof}

At these two extremes two distinct subspaces in $\mfk{g}$ appear as the reduced spectra of cohomology.  For the equivariantization $\mrm{rep}(u_q(\mfk{g})^\ast)^{G\times G^{op}}$ the unit is the equivariant $u_q(\mfk{g})^\ast$-representation $\1=k[G]$, and one can calculate that the specturm of cohomology is the sum of the positive and negative nilpotent subalgebras $\mfk{n}_\pm\subset \mfk{g}$,
\[
\opn{Spec}H^\b(\mrm{rep}(u_q(\mfk{g})^\ast)^{G\times G^{op}},\1)_\mrm{red}=\mfk{n}_+\times \mfk{n}_-,
\]
using~\cite{ginzburgkumar93}.  For $u_q(\mfk{g})$,~\cite{ginzburgkumar93} tells us directly that the spectrum is the nilpotent cone $\mrm{Nil}(\mfk{g})$.

\begin{remark}\label{rem:q_group}
One can consider the above examples, or the example of the quantum group $u_q(\mfk{g})$ versus the product $u_q(\mfk{b}_+)\ot u_q(\mfk{b}_-)$~\cite[Remark 7.14.12]{egno15}, to see that more refined properties of cohomology than Krull dimension are in general not preserved under duality.  Indeed, the spectrum of cohomology for $u_q(\mfk{g})$, which is the nilpotent cone, is highly singular in general, while the spectrum of cohomology for $u_q(\mfk{b}_+)\ot u_q(\mfk{b}_-)$ is an affine space.
\end{remark}

\subsection{Cohomology of dynamical quantum groups}

From Proposition~\ref{prop:Cduals} we derive finite generation of cohomology for dynamical quantum groups.

\begin{corollary}
Given any Belavin-Drinfeld triple $T$ for $\mfk{g}$, the associated dynamical quantum group $u_q(\mfk{g},T)$ has finitely generated cohomology.  Furthermore, the Krull dimension is bounded as
\[
\Kdim \mrm{rep}\left(u_q(\mfk{g},T)\right)\leq 2\dim \mfk{g}/\mfk{h}.
\]
\end{corollary}

\begin{proof}
We have the equivalence $\mrm{rep}(u_q(\mfk{g},T))\simeq \mrm{rep}(u_q(\mfk{g}))^\ast_{\msc{M}_T}$ and apply Proposition~\ref{prop:Cduals}, or Theorem~\ref{thm:char0}.  Our calculation of dimension comes from the fact that $\Kdim\rep(u_q(\mfk{g}))=\dim\mrm{Nil}(\mfk{g})=\dim \mfk{g}/\mfk{h}$.
\end{proof}

At the moment we do not have an explicit understanding of the spectrum of cohomology at general $T$.  From the results at the extreme ends, where the spectra are $\mfk{n}_+\times \mfk{n}_-$ and $\mrm{Nil}(\mfk{g})$, we expect that the spectra should admit a general description in terms of purely Lie theoretic data, and should be of constant dimension $\dim\mfk{g}/\mfk{h}$.  In particular, we expect that the parabolics associated to the $\Gamma_i$ will appear in such a formulation.

\section{Finite type braided categories I: a preliminary discussion}
\label{sect:braidedI}

Take $k$ algebraically closed and of characteristic $0$.  This assumption will hold generically for the remainder of this work, although in Sections~\ref{sect:ses} and~\ref{sect:Gext} the assumption can be relaxed.  By our base change result, Lemma~\ref{lem:basechange}, the effective restriction here is that $\msc{C}$ should be a tensor category over a base field of characteristic $0$.
\par

We now turn our focus from specific occurrences of weak stability to a discussion of the weak stability conjecture in general, with an emphasis on tensor categories which admit a braiding.  Recall the following basic definition of M\"uger~\cite{muger03II}.

\begin{definition}
For a braided tensor category $\C$, the M\"uger center $Z_2(\C)$ is the full subcategory generated by all objects $V$ in $\msc{C}$ for which the square brading $c_{-,V}c_{V,-}\in \mrm{End}_\mrm{Fun}(V\ot -)$ is the identity.
\end{definition}

The main purpose of the present section is to outline a plan for addressing the cohomology of braided categories in characteristic $0$.  We outline our general approach in Section~\ref{sect:1237}, after establishing some background material.  We carry out the plan outlined here in the remaining sections of the paper.

\subsection{M\"uger centers and (super-)Tannakian categories}
\label{sect:Mu}

By Deligne's theorem~\cite{deligne90,deligne02}, any (finite) symmetric braided tensor category $\msc{E}$ (over an algebraically closed field of characteristic $0$) is equivalent to the category $\mrm{srep}(SG)$ of super-representations of a finite supergroup $SG$.  In Hopf theoretic terms, $\mrm{srep}(SG)$ is the symmetric category of representations of a triangular Hopf algebra of the form $\mrm{Wedge}(V)\rtimes G$, where $G$ is a finite group, $V$ is a $G$-representation of $(z,1)$-skew primitives, where $z$ in the center of $G$ with $z^2=1$, and the $R$-matrix is $R_z=\frac{1}{2}(1+z\ot 1+1\ot z-z\ot z)$ (cf.~\cite{andrusetingofgelaki01}).

If we suppose that $\msc{E}$ is semisimple then $V$ must vanish, and $\msc{E}$ is equivalent to the symmetric category $\rep(G,R_z)$.

\begin{remark}
The braiding induced by $R_z$ is just the usual signed super-swap, where the even and odd parity components of a representation are given by the $\pm 1$ eigenspaces for the action of $z$.  Note that when $z=1$ this braiding is trivial, as all representations are concentrated in even degree.
\end{remark}

When $z$ is not $1$, the existence of objects in $\mrm{srep}(SG)$ with self braiding $c_{VV}=-id_{V\ot V}$ obstructs the existence of a symmetric fiber functor to $Vect$.  So if $\msc{E}$ admits a symmetric fiber functor to $Vect$ then $\msc{E}$ is isomorphic to $\rep(G)$ for some finite group $G$, with the trivial braiding.  In particular, $\msc{E}$ must be semisimple in this case.
\par

A symmetric category $\msc{E}$ is called {\it Tannakian} if it admits a symmetric fiber functor to $Vect$, and non-Tannakian, or {\it super-Tannakian}, otherwise.  All super-Tannakian categories considered in this work will be, or will be proved to be, semisimple, i.e.\ of the form $\msc{E}\cong \rep(G,R_z)$.
\par

In the sections that follow we address weak stability of cohomology for braided categories whose M\"uger centers are semisimple, which we refer to colloquially as categories with a semisimple degeneracy.  Our investigations bifurcate into an analysis of categories with Tannakian M\"uger center, and categories with semisimple, super-Tannakian, M\"uger center.

\subsection{Relative centers}

Consider a tensor category $\msc{C}$ and a surjective tensor functor $F:\msc{C}\to \msc{D}$.  The functor $F$ is faithful, by Lemma~\ref{lemma:surj_func_faithful}.  For convenience, we identify $\msc{C}$ with its image in $\msc{D}$.
\par

We have that $\msc{D}$ is an exact module category over both $\msc{C}$ and $\msc{C}\bt\msc{D}^{\mrm{cop}}$.  Following~\cite{gelakinaidunikshych09}, we define the {\it relative center} $Z_\C(\D)$ as follows:
\begin{itemize}
\item The objects of $Z_\C(\D)$ are pairs $(V, \gamma_V),$ where $V$ is an object of $\D$ and $\gamma_V: V\otimes - \overset{\cong}\to -\otimes V$ is a natural isomorphism between the functors $V\ot-,\ -\ot V:\msc{C}\to \msc{D}$ which is compatible with the associators and unit in the expected way.  For any such pair $(V,\gamma_V)$, we call $\gamma_V$ the half-braiding.
\item Morphisms in $Z_\C(\D)$ are maps $f:V\to W$ in $\msc{D}$ which are compatible with the chosen half braidings $\gamma_{V,-}$ and $\gamma_{W,-}$ in the sense that the diagram
\[
\xymatrix{
V\ot X\ar[r]^{\gamma_{V,X}}\ar[d]_{f\ot id} & X\ot V\ar[d]^{id\ot f}\\
W\ot X\ar[r]^{\gamma_{W,X}} & X\ot W
}
\]
commutes for all $X\in\C$.
\item The tensor product $(V,\gamma_{V,-})\ot (W,\gamma_{W,-})$ is the object $V\ot W$ in $\msc{D}$ with the obvious half braiding.
\end{itemize}

We note, as in~\cite{gelakinaidunikshych09}, that there is an equivalence $(\C\bt \D^{\mrm{cop}})^*_\D\overset{\sim}\to Z_\C(\D)$ which sends a $\C\bt\D^{op}$-module endofunctor $L$ of $\D$ to the pair consisting of the object $V=L(\1)$ along with the half-braiding
\[
\gamma_{V,X}:V\ot X=L(\1)\ot X\overset{\cong}\to L(X)\overset{\cong}\to X\ot L(\1)=X\ot V,
\]
for $X\in \C$. We have a pair of functors
\begin{equation}\label{eq:1189}
\xymatrixrowsep{3mm}
\xymatrix{
Z(\C)\ar[dr]^(.45){q} & & Z(\D)\ar[dl]_(.45){r},\\
 & Z_\C(\D) &
}
\end{equation}
where $q$ is the map taking a pair of an object in $\C$ with a half braiding to the same object in $\D$ with the same half braiding, and $r$ restricts the half braiding on an object in $\D$ to objects in $\C$.

\subsection{Relative centers and de-equivariantization}

Let $\C$ be a braided tensor category in characteristic $0$ and consider a central embedding $\rep(G)\to \msc{C}$ which has image in the M\"uger center of $\msc{C}$.  In this case the de-equivariantization $\C_G$ inherits a braiding from $\C$.
\par

Since any de-equivariantization functor is surjective, we may consider the relative center $Z_{\C}(\C_G)$.

\begin{lemma}\label{lem:1219}
For any central embedding $\mrm{rep}(G)\to \msc{C}$ there is a tensor equivalence $Z_\msc{C}(\msc{C}_G)\simeq Z(\msc{C})_G$, where the de-equivariantization of $Z(\msc{C})$ is induced by the braided inclusion $\rep(G)\to \C\to Z(\msc{C})$.
\end{lemma}

The result should be well-known, and so we only sketch the proof.
\begin{proof}[Sketch proof]
There is a tensor functor $F:Z_\msc{C}(\msc{C}_G)\to Z(\msc{C})_G$ which sends an object $(V,\mrm{act}_V,\gamma_V)$ in the relative center, where $\mrm{act}_V:\O\ot V\to V$ is the action of $\O=\O(G)$, to the object $(V,\gamma_V)$ in $Z(\msc{C})$ with the action $\mrm{act}_V$.  The fact that $(V,\mrm{act}_V,\gamma_V)$ is an object in $Z_\msc{C}(\msc{C}_G)$ exactly implies that $\mrm{act}_V:\O\ot V\to V$ is a morphism in $Z(\msc{C})$.  So, we see that $((V,\gamma_V),\mrm{act}_V)$ is indeed an object in $Z(\msc{C})_G$.  One similarly constructs the inverse $F^{-1}$ to see that $F$ is an equivalence.
\end{proof}

\begin{corollary}\label{cor:relcen_finitetype}
For any braided tensor category $\C$ and any central embedding $\mrm{rep}(G)\to \msc{C}$, the center $Z(\C)$ is of finite type if and only if the relative center $Z_\C(\C_G)$ is of finite type. Furthermore $\Kdim Z(\msc{C})=\Kdim Z_\C(\msc{C}_G)$.
\end{corollary}

\begin{proof}
Apply Theorem~\ref{thm:HGC} and Lemma~\ref{lem:1219}.
\end{proof}

\subsection{The approach of Section~\ref{sect:braidedII}}
\label{sect:1237}

Consider a braided tensor category $\msc{C}$ and a central embedding $\rep(G)\to \msc{C}$ with image equal to the M\"uger center $Z_2(\msc{C})$ in $\msc{C}$.  Suppose $\msc{C}$ is of finite type.  Our goal is to show that the Drinfeld center $Z(\msc{C})$ is of finite type, from which we deduce the finite type property for arbitrary duals $\msc{C}^\ast_\msc{M}$.  However, we cannot access the center directly here, as was the case in the non-degenerate setting.  We therefore employ the relative center $Z_\msc{C}(\msc{C}_G)$ as an intermediary between $\msc{C}$ and $Z(\msc{C})$, in the manner outlined below.
\par

Recall the following essential theorem from~\cite{dgno10}, which is implicit in the initial works of \cite{bruguieres00,muger00}.

\begin{theorem}[{\cite{bruguieres00,muger00,dgno10}}]\label{modular deequivariantization}
Suppose $\rep(G)\to\C$ is a braided equivalence onto $Z_2(\msc{C})$.  Then the de-equivariantization $\C_G$ is non-degenerate.
\end{theorem}

Suppose as before that $\msc{C}$ is of finite type.  By transfer of cohomology along de-equivariantization, Proposition~\ref{prop:Cduals}, and the above theorem, we have that $Z(\msc{C}_G)$ is of finite type and of Krull dimension $\Kdim Z(\msc{C}_G)=2 \Kdim\msc{C}$.
\par

As in~\eqref{eq:1189}, we have the diagram
\[
\xymatrixrowsep{3mm}
\xymatrix{
Z(\C)\ar[dr]^(.45){q} & & Z(\C_G)\ar[dl]_(.45){r}\\
 & Z_\C(\C_G) & .
}
\]
Recall that we want to verify the finite type property for $Z(\msc{C})$.  Since $Z_{\msc{C}}(\msc{C}_G)$ is a de-equivariantization of $Z(\msc{C})$, we can transfer cohomology {\it up} from $Z_\msc{C}(\msc{C}_G)$ to $Z(\msc{C})$.  So $Z(\msc{C})$ is of finite type whenever the relative center is of finite type.  Therefore, we need to find a way to transfer cohomology {\it down} from the usual center $Z(\msc{C}_G)$ to $Z_\msc{C}(\msc{C}_G)$.
\par

We achieve the desired transfer of cohomology along $r$ by extending to a (relative) exact sequence
\[
Z(\msc{C}_G)\overset{r}\to Z_\msc{C}(\msc{C}_G)\to \msc{Q},
\]
where $\msc{Q}$ is a yet to be defined monoidal category, and employing a generalized Lyndon-Hochschild-Serre spectral sequence to obtain the cohomology of $Z_\msc{C}(\msc{C}_G)$ from the cohomology of $Z(\msc{C}_G)$.  In Section~\ref{sect:ses} below we provide the necessary background on short exact sequences, and in Section~\ref{sect:braidedII} we realize the above outline to find that $Z(\msc{C})$ is of finite type whenever $\msc{C}$ is of finite type and has Tannakian M\"uger center.  In Section~\ref{sect:braidedIII} we enhance the above approach to address categories $\msc{C}$ with possibly non-Tannakian, semisimple, M\"uger center.

\section{The Lyndon-Hochschild-Serre spectral sequence for exact sequences of categories}
\label{sect:ses}

For this section, we let $k$ be algebraically closed and of arbitrary characteristic.  We will, however, only apply the results of this section to cases in which $\mrm{char}(k)=0$.
\par

We show that any exact sequence of tensor categories
\[
\msc{B}\to\msc{C}\to\msc{D}\bt\End_k(\msc{M}),
\]
in the generalized sense of Etingof and Gelaki~\cite{etingofgelaki17} (cf. Brugui\`eres and Natale~\cite{bruguieresnatale11}), gives rise to a spectral sequence
\[
``H^i(\msc{B},H^j(\msc{D},V))"\ \Rightarrow\ H^{i+j}(\msc{C},V).
\]
In order to do this, we must first specify what exactly is on the left hand side of the above equation.  The issue here is that the $\msc{D}$-invariants functor is a functor from $\msc{D}$ to $Vect$, not one from $\msc{C}$ to $\msc{B}$.  So we define an appropriate ``$\msc{D}$-invariants functor" in this setting, which we denote $\mfk{H}^0_\msc{C}(\msc{D},-):\msc{C}\to \msc{B}$.  Of course, for exact sequences of finite groups our spectral sequence reduces to the familiar Lyndon-Hochschild-Serre spectral sequence.

\subsection{Exact sequences (following~\cite{etingofgelaki17})}

Recall that a $k$-linear category is, for us, an abelian category enriched over $Vect$.  We have the standard defintion,

\begin{definition}
For a finite $k$-linear category $\msc{M}$ we let $\End_k(\msc{M})$ denote the $k$-linear monoidal category of right exact $k$-linear endofunctors of $\msc{M}$.
\end{definition}

\begin{remark}
The finiteness assumption here is inessential, but appropriate for our analysis.
\end{remark}

The category $\End_k(\msc{M})$ is monoidal under composition and has unit $id_\msc{M}$.  If we write $\msc{M}=\rep(R)$, for a finite dimensional algebra $R$, then the functor $\mathrm{bimod}(R)\to \End_k(\msc{M})$, $M\mapsto M\ot_R-$ is an equivalence of $k$-linear monoidal categories, by classical Morita theory.  In particular, $\End_k(\msc{M})$ is Artinian and Noetherian.
\par

Consider an arbitrary tensor category $\msc{D}$ and the corresponding Deligne product $\msc{D}\bt \End_k(\msc{M})$.  We embed $\End_k(\msc{M})$ in $\msc{D}\bt \End_k(\msc{M})$ as $X\mapsto \mbf{1}_{\msc{D}}\ot X$.  The following lemma will help us make sense of the notion of a ``normal" map to $\msc{D}\bt \End_k(\msc{M})$.

\begin{lemma}\label{lem:134}
Every object $W$ in $\msc{D}\bt\End_k(\msc{M})$ admits a unique maximal subobject $W'$ in (the image of) $\End_k(\msc{M})$.  Furthermore, $W'$ can be defined as the maximal subobject in $W$ which admits a surjective map $\1\ot X\to W$ from an object $X$ in $\End_k(\msc{M})$.
\end{lemma}

\begin{proof}
If any such maximal object exists it provides a universal morphism $\mbf{1}\ot X(W)\to W$ from an object $X(W)$ in $\End_k(\msc{M})$.  Consider now the maximal subobject $W'\subset W$ which is the image of some $\1\ot X\to W$.  Such a $W'$ exists as $\msc{D}\bt\End_k(\msc{M})$ is Noetherian and for any two $\1\ot X\to W$ and $\1\ot Y\to W$ the coproduct map is expressible as a map from $\1\ot (X\oplus Y)$.
\par

We claim that the kernel of any projection $\1\ot X\to W'$ is of the form $\1\ot Y$ for some subobject $Y$ in $X$.  This is easy to see if we adopt a Morita equivalence $\msc{D}\cong \rep(\Pi)$ for a basic algebra $\Pi$, and $\msc{M}=\rep(R)$, so that $\msc{D}\bt \End_k(\msc{M})\cong\rep(\Pi\ot (R\ot R^{op}))$.  Under such an equivalence $\1$ is identified with a $1$-dimensional representation $\1=k e$ for $\Pi$, and the kernel of any map $\1\ot X\to W'$ is a span of vectors $\sum_i e\ot x_i=e\ot(\sum_i x_i)$.  Hence the kernel is of the form $\1\ot Y$.  By exactness of the tensor product this gives $W'\cong \1\ot X(W)$ for $X(W)=X/Y$.
\end{proof}

The following definitions is due to Etingof and Gelaki~\cite{etingofgelaki17}, (cf.~the earlier work of Brugui\`eres and Natale~\cite{bruguieresnatale11}).

\begin{definition}[\cite{etingofgelaki17}]
Let $\msc{M}$ be a finite $k$-linear abelian category, and consider tensor categories $\msc{C}$ and $\msc{D}$.  We say an exact monoidal functor $F:\msc{C}\to \msc{D}\bt\End_k(\msc{M})$ is normal relative to $\msc{M}$ if for each $V$ in $\msc{C}$ there is a subobject $V_{\msc{D}\text{-}\mrm{triv}}\subset V$ such that
\[
F(V_{\msc{D}\text{-}\mrm{triv}})\subset F(V)
\]
is the maximal subobject in $F(V)$ contained in $\End_k(\msc{M})$.
\end{definition}

\begin{definition}[\cite{etingofgelaki17}]
Let $\msc{M}$ be a finite $k$-linear abelian category.  Let 
\begin{equation}\label{eq:seq}
\msc{B}\overset{i}\to \msc{C}\overset{F}\to \msc{D}\boxtimes \End_k(\msc{M})
\end{equation}
be a composition of a tensor functor $i:\msc{B}\to \msc{C}$ with an exact monoidal functor $F:\msc{C}\to \msc{D}\bt\End_k(\msc{M})$.  We say the sequence~\eqref{eq:seq} is an exact sequence with respect to $\msc{M}$ if the following conditions hold:
\begin{enumerate}
\item[(a)] The tensor functor $i$ is a fully faithful embedding.
\item[(b)] The monoidal functor $F$ is normal relative to $\msc{M}$, and surjective.
\item[(c)] The category $\msc{B}$ is the kernel of $F$ relative to $\msc{M}$, i.e.\ $\msc{B}$ is the full subcategory in $\msc{C}$ consisting of all objects which map into $\End_k(\msc{M})$ under $F$.
\item[(d)] The $\msc{B}$-action on $\msc{M}$ induced by the functor $\msc{B}\to \End_k(\msc{M})$ makes $\msc{M}$ into an exact indecomposable $\msc{B}$-module category.
\end{enumerate}
\end{definition}

\begin{remark}
When convenient, we identify $\msc{B}$ with its image $i(\msc{B})$ in $\msc{C}$. Indeed, we already did this in the previous definition in item (c).
\end{remark}
\begin{remark}
It is natural, from the algebraic perspective, to only consider exact sequences with respect to the trivial module category $\msc{M}=Vect$.  This is the perspective taken in the initial work of Brugui\`eres and Natale~\cite{bruguieresnatale11}.  However, if we take the {\it dual} of such a nice sequence $\msc{Q}\to \msc{R}\to \msc{S}$ with respect to an exact indecomposable $\msc{S}$-module category $\msc{N}$, which is a completely natural thing to do, we arrive at a sequence
\[
\msc{S}^\ast_\msc{N}\to \msc{R}^\ast_\msc{N}\to \msc{Q}^\ast_{Vect}\bt\End_k(\msc{N}).
\]
This new sequence will be exact with respect to $\msc{M}$~\cite[Theorem\ 4.1]{etingofgelaki17}, but generally not exact in the sense of~\cite{bruguieresnatale11}.  This duality property is essential to our study.
\end{remark}

\subsection{Homological properties of the relative invariants}

\begin{lemma}\label{lem:156}
If $\msc{B}\overset{i}\to \msc{C}\overset{F}\to \msc{D}\boxtimes \End_k(\msc{M})$ is an exact sequence relative to $\msc{M}$, then $F$ is faithful.
\end{lemma}

\begin{proof}
As is explained in~\cite{etingofgelaki17}, we have 
\[
\msc{D}\bt\End_k(\msc{M})\cong \End_{\msc{D}^\mrm{cop}}(\msc{D})\bt\End_k(\msc{M}).
\]
It follows that the map $F:\msc{C}\to \msc{D}\bt \End_k(\msc{M})$ produces an action of $\msc{C}$ on $\msc{D}\bt \msc{M}$ which commutes with the action of $\msc{D}^\mrm{cop}$ on $\msc{D}$.  Hence $F$ is naturally given by restricting the codomain of an action map
\begin{equation}\label{eq:168}
\msc{D}^\mrm{cop}\bt\msc{C}\to \End_k(\msc{D}\bt\msc{M}).
\end{equation}
By~\cite[Corollary 2.5]{etingofgelaki17}, $\msc{D}\bt\msc{M}$ is exact over $\msc{D}^\mrm{cop}\bt\msc{C}$, and hence the functor~\eqref{eq:168} is faithful.  As a consequence we find that the restriction $\msc{C}\to \End_k(\msc{D}\bt\msc{M})$ to $\msc{C}\subset \msc{D}^\mrm{cop}\bt\msc{C}$ is faithful, as is its factoring $F:\msc{C}\to \msc{D}\bt \End_k(\msc{M})$ through $\End_{\msc{D}^\mrm{cop}}(\msc{D})\bt\End_k(\msc{M})$.
\end{proof}

\begin{lemma}\label{lem:172}
Let $F:\msc{C}\to \msc{D}\boxtimes \End_k(\msc{M})$ be a faithful monoidal functor which is normal relative to $\msc{M}$.  Then the following properties hold:
\begin{enumerate}
\item[(i)] Every object $V$ in $\msc{C}$ admits a subobject $V_{\msc{D}\text{-}\mrm{triv}}\subset V$ such that $F(V_{\msc{D}\text{-}\mrm{triv}})$ is the unique maximal subobject in $F(V)$ which lies in $\End_k(\msc{M})$.  Furthermore, $V_{\msc{D}\text{-}\mrm{triv}}$ is determined uniquely by this property.
\item[(ii)] For any map $f:V\to W$ in $\msc{C}$, the restriction $f|_{V_{\msc{D}\text{-}triv}}$ factors (uniquely) through $W_{\msc{D}\text{-}\mathrm{triv}}$.
\end{enumerate}
That is to say, the operation $V\mapsto V_{\msc{D}\text{-}\mrm{triv}}$ is an endofunctor of $\msc{C}$ with image in the kernel of $F$.
\end{lemma}

\begin{proof}
(i) Normality ensures that such a $V_{\msc{D}\text{-}\mrm{triv}}$ exists.  For uniqueness, suppose another submodule $V'\subset V$ is mapped to an object in $\End(\msc{M})$ under $F$, and consider the sequence
\[
F(V')\to F(V)\to F(V)/F(V_{\msc{D}\text{-}\mrm{triv}}).
\]
By maximality of $F(V_{\msc{D}\text{-}\mrm{triv}})$, the above sequence is $0$.  By faithfulness of $F$ the sequence $V'\to V\to V/V_{\msc{D}\text{-}\mrm{triv}}$ is also $0$.  Hence $V'\subset V_{\msc{D}\text{-}\mrm{triv}}$.  This implies uniqueness of $V_{\msc{D}\text{-}\mrm{triv}}$.  (ii) Functoriality of $(-)_{\msc{D}\text{-}\mrm{triv}}$ with respect to morphisms follows by a similar argument.
\end{proof}

Lemmas~\ref{lem:156} and~\ref{lem:172} together tell us that any exact sequence
\begin{equation}\label{eq:184}
\msc{B}\overset{i}\to \msc{C}\overset{F}\to \msc{D}\boxtimes \End_k(\msc{M})
\end{equation}
produces a functor $\msc{C}\to \msc{B}$ which assigns to any object in $\msc{C}$ its maximal $\msc{D}$-trivial subobject.  The fact that this subobject lies in $\msc{B}$ follows from the identification $\msc{B}=\ker(F)$.

\begin{definition}
For an exact sequence~\eqref{eq:184}, relative to an exact $\msc{B}$-module category $\msc{M}$, we define the functor $\mfk{H}^0_{\msc{C}}(\msc{D},-)$ by taking
\[
\mfk{H}^0_{\msc{C}}(\msc{D},-):\msc{C}\to \msc{B},\ \ V\mapsto V_{\msc{D}\text{-}\mrm{triv}},
\]
where $V_{\msc{D}\text{-}\mrm{triv}}$ is as in Lemma~\ref{lem:172}.  We call $\mfk{H}^0_\msc{C}(\msc{D},-)$ the {\it $\msc{D}$-relative invariants}.
\end{definition}

\begin{lemma}\label{lem:properties}
For an exact sequence~\eqref{eq:184}, the relative invariants functor $\mfk{H}^0_\msc{C}(\msc{D},-)$ has the following properties:
\begin{enumerate}
\item[(i)] $\mfk{H}^0_\msc{C}(\msc{D},-)$ is left exact.
\item[(ii)] $\mfk{H}^0_{\msc{C}}(\msc{D},-)$ is exact when $\msc{D}$ is a fusion category.
\item[(iii)] $\mfk{H}^0_{\msc{C}}(\msc{D},-)$ sends injectives in $\msc{C}$ to injectives in $\msc{B}$.
\item[(iv)] There is a natural identification of functors
\[
\Hom_\msc{B}(\1_{\msc{B}},\mfk{H}^0_\msc{C}(\msc{D},-))=\Hom_{\msc{C}}(\1_\msc{C},-).
\]
\end{enumerate}
\end{lemma}

\begin{proof}
(i) For any exact sequence $0\to L\to M\to N\to 0$ it is clear that $\mfk{H}^0_\msc{C}(\msc{D},L)\to \mfk{H}^0_\msc{C}(\msc{D},M)$ is injective, via the maximal image interpretation of Lemma~\ref{lem:134} and the fact that $F$ is faithfully exact~\cite[Pp 1195]{{etingofgelaki17}}.  Now, after applying $F$, the kernel of $F\mfk{H}^0_\msc{C}(\msc{D},M)\to F\mfk{H}^0_\msc{C}(\msc{D},N)$ is an object of the form $\1\ot Z$, which necessarily factors through $F(L)$, and hence lies in $F\mfk{H}^0_\msc{C}(\msc{D},L)$ by Lemma~\ref{lem:134}.  So the kernel of $F\mfk{H}^0_\msc{C}(\msc{D},M)\to F\mfk{H}^0_\msc{C}(\msc{D},N)$ is contained in $F\mfk{H}^0_\msc{C}(\msc{D},L)$.  The opposite containment is immediate, and we therefore have that $F\mfk{H}^0_\msc{C}(\msc{D},L)$ is mapped isomorphically onto the kernel of $F\mfk{H}^0_\msc{C}(\msc{D},M)\to F\mfk{H}^0_\msc{C}(\msc{D},N)$.  Faithfulness of $F$ implies that $\mfk{H}^0_\msc{C}(\msc{D},L)$ is the kernel of $\mfk{H}^0_\msc{C}(\msc{D},M)\to \mfk{H}^0_\msc{C}(\msc{D},N)$, as desired.
\par

(ii) When $\msc{D}$ is fusion, we may enumerate the simples $\{V_0,V_1,\dots, V_n\}$ with $V_0=\1$.  In this case all objects of the Deligne product $\msc{D}\bt \End_k(\msc{M})$ are direct sums of objects of the form $V_i\ot X$, and there are no nonzero maps between $V_i\ot X$ and $V_j\ot Y$ for $i\neq j$.  Therefore taking the trivial summand $V_{\mrm{triv}}\ot X$, which provides the maximal subobject in $\End_k(\msc{M})$, is an exact operation.  It follows that for an exact sequence $0\to L\to M\to N\to 0$ in $\msc{C}$ the sequence
\[
0\to \mfk{H}^0_\msc{C}(\msc{D},L)\to \mfk{H}^0_\msc{C}(\msc{D},W)\to \mfk{H}^0_\msc{C}(\msc{D},N)\to 0
\]
is exact in $\msc{B}$, since its image under $F$ is exact.
\par

(iii) The functor $\mfk{H}^0_\msc{C}(\msc{D},-)$ can alternatively be defined as the functor which sends an object $V$ to the universal map $\mfk{H}^0_\msc{C}(\msc{D},V)\to V$ from an object in $\msc{B}$.  So for any object $L$ in $\msc{B}$ we have $\Hom_{\msc{C}}(L,V)=\Hom_{\msc{B}}(L,\mfk{H}^0_\msc{C}(\msc{D},V))$.
\par

Suppose that $V$ is injective in $\msc{C}$ and that $0\to L\to M\to N\to 0$ is an exact sequence in $\msc{B}$.  Then we have the diagram
\[
{\small
\xymatrixcolsep{5mm}
\xymatrix{
& \Hom_{\msc{B}}(N,\mfk{H}^0_\msc{C}(\msc{D},V))\ar[r]\ar[d]^\cong & \Hom_{\msc{B}}(M,\mfk{H}^0_\msc{C}(\msc{D},V)) \ar[r]\ar[d]^\cong & \Hom_{\msc{B}}(L,\mfk{H}^0_\msc{C}(\msc{D},V))\ar[d]^\cong &\\
0\ar[r]& \Hom_{\msc{C}}(N,V)\ar[r] & \Hom_{\msc{C}}(M,V) \ar[r] & \Hom_{\msc{C}}(L,V)\ar[r] & 0,
}}
\]
and conclude that the top row is exact.  Whence $\mfk{H}^0_\msc{C}(\msc{D},V)$ is seen to be injective in $\msc{B}$.
\par

(iv) Since $\1_\msc{C}=i(\1_\msc{B})$ is an object in (the image of) $\msc{B}$, the universal map perspective employed in (iii) gives a natural identification
\[
\Hom_\msc{C}(\1_\msc{C},V)= \Hom_\msc{B}(\1_\msc{B},\mfk{H}^0_\msc{C}(\msc{D},V)),
\]
for arbitrary $V$ in $\msc{C}$.
\end{proof}

\subsection{Lyndon-Hochschild-Serre spectral sequence}  Consider as before a relative exact sequence~\eqref{eq:184}.  By point (i) of Lemma~\ref{lem:properties} we can derive the functor $\mfk{H}^0_\msc{C}(\msc{D},-):\msc{C}\to \msc{B}$ to get a triangulated functor
\[
\mrm{R}\mfk{H}^0_\msc{C}(\msc{D},-):\mrm{D}^b(\msc{C})\to \mrm{D}^+(\msc{B}).
\]
We denote the corresponding $i$-th derived functor
\[
\mfk{H}^i_\msc{C}(\msc{D},-)=\mrm{R}^i\mfk{H}^0_\msc{C}(\msc{D},-).
\]
Each $\mfk{H}^i_\msc{C}(\msc{D},-)$ is a functor taking values in the category $\msc{B}$.
\par

The object $\mrm{R}\mfk{H}^0_\msc{C}(\msc{D},\mbf{1})$ admits an algebra structure in $\mrm{D}^+(\msc{B})$ via the usual yoga.  Namely, one takes an injective resolution $\1_\msc{C}\to I$, so that $\mrm{R}\mfk{H}^0_\msc{C}(\msc{D},\mbf{1})=\mfk{H}^0_\msc{C}(\msc{D},I)$, and chooses a homotopy equivalence $I\ot I\to I$ to get
\[
\mrm{mult}_{\mrm{R}\mfk{H}^0_\msc{C}(\msc{D},\mbf{1})}:
\mfk{H}^0_\msc{C}(\msc{D},I)\ot \mfk{H}^0_\msc{C}(\msc{D},I)\to \mfk{H}^0_\msc{C}(\msc{D},I\ot I)\to \mfk{H}^0_\msc{C}(\msc{D},I).
\]
We similarly get an action of $\mrm{R}\mfk{H}^0_\msc{C}(\msc{D},\mbf{1})$ on $\mrm{R}\mfk{H}^0_\msc{C}(\msc{D},V)$ for each $V$ in $\msc{C}$.

\begin{proposition}\label{prop:ss}
Let $\msc{B}\to \msc{C}\to \msc{D}\bt\End_k(\msc{M})$ be an exact sequence with respect to $\msc{M}$.  There is a convergent multiplicative spectral sequence
\[
E^{i,j}_2=H^i(\msc{B},\mfk{H}^j_\msc{C}(\msc{D},\mbf{1}))\ \Rightarrow\ H^{i+j}(\msc{C},\mbf{1}),
\]
and for each $V$ in $\msc{C}$ there is a spectral sequence
\[
{^V\! E}^{i,j}_2=H^i(\msc{B},\mfk{H}^j_\msc{C}(\msc{D},V))\ \Rightarrow\ H^{i+j}(\msc{C},V)
\]
which is a module over $E^{\b,\b}_\ast$.
\end{proposition}

\begin{proof}
By point (iii) of Lemma~\ref{lem:properties}, the proposed spectral sequence arrises as a Grothendieck spectral sequence.  The multiplicative properties follow by~\cite[Proposition 5.1]{friedlandernegron18}.
\end{proof}

We only use the above spectral sequence in the case in which $\msc{D}$ is fusion.  In this case the spectral sequence degenerates to give a $k$-algebra identification
\[
H^\b(\msc{B},\mbf{1}_\msc{B})=H^\b(\msc{B},\mfk{H}^0_\msc{C}(\msc{D},\mbf{1}_\msc{C}))=H^\b(\msc{C},\mbf{1}_\msc{C}),
\]
and an identification of graded $H^\b(\msc{B},\1)$-modules
\[
H^\b(\msc{B},\mfk{H}^0_\msc{C}(\msc{D},V))=H^\b(\msc{C},V).
\]
Both of these identifications are induced by the exact embedding $i:\msc{B}\to \msc{C}$.  We record these findings in a corollary.

\begin{corollary}\label{cor:ss}
Let $\msc{B}\to \msc{C}\to \msc{D}\bt\End_k(\msc{M})$ be an exact sequence with respect to $\msc{M}$ and suppose that $\msc{D}$ is a fusion category.  Then there is a natural identification of $k$-algebras $H^\b(\msc{B},\mbf{1})=H^\b(\msc{C},\mbf{1})$.  Furthermore, for any object $V$ in $\msc{C}$, there is an identification of $H^\b(\msc{C},\mbf{1})$-modules $H^\b(\msc{B},\mfk{H}^0_\msc{C}(\msc{D},V))=H^\b(\msc{C},V)$, where the cohomology of $\msc{C}$ acts on the right hand side of this equality via its identification with the cohomology of $\msc{B}$.
\end{corollary}

\section{Finite type braided categories II: Tannakian M\"uger center}
\label{sect:braidedII}

We fix $k$ an algebraically closed field of characteristic $0$. We follow the outline of Section~\ref{sect:1237} to verify stability of cohomology for finite type braided tensor categories with Tannakian M\"uger center.  We describe some practical methods for determining whether or not the M\"uger center of a given braided category is Tannakian in Section~\ref{sect:tann_check}.

\subsection{De-equivariantization and doubling for braided tensor categories}

Let $\rep(G)\to \msc{C}$ be a central embedding, with $G$ a finite group, and consider the de-equivariantization $\msc{C}_G$.  We then have the exact sequence
\[
\rep(G)\to \msc{C}\to \msc{C}_G
\]
with respect to $Vect$ and tensor with $\msc{C}_G^\mrm{cop}$ on the right to get another exact sequence
\[
\rep(G)\to \msc{C}\bt \msc{C}_G^\mrm{cop}\to \msc{C}_G\bt\msc{C}_G^\mrm{cop}
\]
with respect to $Vect$.  By~\cite[Theorem\ 4.1]{etingofgelaki17} we may take the dual with respect to the indecomposable exact $\msc{C}_G$-bimodule category $\msc{C}_G$ to get another exact sequence
\begin{equation}\label{eq:323}
Z(\msc{C}_G)\overset{r}\to Z_{\msc{C}}(\msc{C}_G)\to Vect^G\bt \End_k(\msc{C}_G),
\end{equation}
now with respect to $\msc{C}_G$.

\begin{theorem}\label{thm:ZCvZC_G}
Suppose $k$ is algebraically closed and of characteristic $0$. Let $\rep(G)\to \msc{C}$ be a central embedding, with $G$ a finite group.  Then the center $Z(\msc{C})$ is of finite type if and only if the center of the de-equivariantization $Z(\msc{C}_G)$ is of finite type.  In this case $\Kdim Z(\msc{C})=\Kdim Z(\msc{C}_G)$.
\end{theorem}

\begin{proof}
Via the exact sequence~\eqref{eq:323} and Corollary~\ref{cor:ss} we see that the relative center $Z_\msc{C}(\msc{C}_G)$ is of finite type if and only if $Z(\msc{C}_G)$ is of finite type, in which case the Krull dimensions agree.  By Lemma~\ref{lem:1219}, $Z_\msc{C}(\msc{C}_G)$ is a de-equivariantization of $Z(\msc{C})$.  Thus we apply Theorem~\ref{thm:HGC} to see that $Z(\msc{C})$ is of finite type if and only if $Z_\msc{C}(\msc{C}_G)$ is of finite type, and again the Krull dimensions agree.
\end{proof}

\subsection{Weak stability of cohomology under duality and doubling}

\begin{theorem}\label{thm:braidedI}
Suppose $\msc{C}$ is braided and of finite type over an algebraically closed field $k$ of characteristic $0$.  Suppose also that the M\"uger center of $\msc{C}$ is Tannakian.  Then, for every exact $\msc{C}$-module category $\msc{M}$, the dual $\msc{C}^\ast_\msc{M}$ is also of finite type.  Furthermore, the Krull dimensions are uniformly bounded
\[
\Kdim\msc{C}^\ast_\msc{M}\leq 2 \Kdim\msc{C}.
\]
\end{theorem}

\begin{proof}
It suffices to show that the center $Z(\msc{C})$ is of finite type and has Krull dimension $\Kdim Z(\msc{C})=2 \Kdim\msc{C}$, by Corollary~\ref{cor:Z2C}.  We adopt a braided identification $\mrm{rep}(G)\cong Z_2(\msc{C})$. 
\par

By Theorem~\ref{thm:HGC}, the de-equivariantization is also of finite type and of the same Krull dimension as $\msc{C}$.  Recall that $\C_G$ is non-degenerate in this case (see Theorem \ref{modular deequivariantization}).  It follows that $Z(\msc{C}_G)$ is of finite type and has Krull dimension
\[
\Kdim Z(\msc{C}_G)=2\Kdim\msc{C}_G=2\Kdim\msc{C},
\]
by Proposition~\ref{prop:Cduals}.  Whence $Z(\msc{C})$ is of finite type and of the prescribed Krull dimension, by Theorem~\ref{thm:ZCvZC_G}.
\end{proof}


\subsection{Practical checks: Tannakian vs super-Tannakian type}
\label{sect:tann_check}

Throughout this subsection we maintain the assumptions $k=\bar{k}$, $\mrm{char}(k)=0$.  From the material of Section~\ref{sect:Mu} one sees that any symmetric tensor category $\msc{E}$ over $k$ has the Chevalley property.  That is to say, the full subcategory $\bar{\msc{E}}\subset\msc{E}$ generated by the simples is a tensor subcategory.

\begin{lemma}\cite[Corollary 2.50(i)]{dgno10}\label{odd Tann}
Let $\msc{E}$ be a symmetric tensor category.  If $\opn{FPdim}(\msc{E})$ is odd, then $\msc{E}$ is Tannakian. 
\end{lemma}

\begin{proof}
Let $\bar{\msc{E}}$ denote the fusion subcategory generated by the simplesof $\msc{E}$.  We have $\bar{\msc{E}}\cong \rep(G,R_z)$, where $z\in G$ is such that $z^2=1$, and $z=1$ if and only if $\msc{E}$ is Tannakian.  If $z\neq 1$ then $G$ has an order $2$ subgroup, and hence $2\mid |G|=\opn{FPdim}(\bar{\msc{E}})$.  Since $\msc{E}$ is integral this implies $2\mid \opn{FPdim}(\msc{E})$, which is explicitly not the case.  Hence $z=1$ and $\msc{E}$ is Tannakian.
\end{proof}

\begin{corollary}\label{cor:2}
Suppose $\msc{C}$ is a braided tensor category of odd Frobenius-Perron dimension.  Then the M\"uger center of $\msc{C}$ is Tannakian.
\end{corollary}

\begin{proof}
Let $\msc{E}$ denote the M\"uger center of $\msc{C}$.  Since the categories are weakly integral, we have that $\opn{FPdim}(\msc{E})|\opn{FPdim}(\msc{C})$.  So $\msc{E}$ must be of odd dimension and therefore Tannakian, by Lemma \ref{odd Tann}.
\end{proof}

The following lemma is standard.  We repeat the proof for the convenience of the reader.

\begin{lemma}\label{lem:1529}
Let $\msc{C}$ be a braided tensor category.  Suppose that $\mrm{Irr}(Z_2(\msc{C}))$ contains no objects $V$ which solve the equation $c_{V,V}=-id_{V\ot V}$, or more generally that $\mrm{Irr}(\msc{C})$ contains no such objects.  Then $Z_2(\msc{C})$ is Tannakian.
\end{lemma}

\begin{proof}
Take $\msc{E}=Z_2(\msc{C})$.  Then the fusion category $\bar{\msc{E}}$ generated by the simples in $\msc{E}$ is of the form $\rep(G,R_z)$, with $\mrm{ord}(z)=2$ whenever $\msc{E}$ is non-Tannakian.  When $z\neq 1$ then $1+z$ is not a unit in the group ring $k[G]$, since $(1-z)(1+z)=0$.  So the quotient representation $k[G]/k[G](1+z)$ is non-zero and any simple summand $V$ of this representation is such that $c_{V,V}=-id_{V\ot V}$.  Thus, if $\mrm{Irr}(\msc{E})$, or more generally $\mrm{Irr}(\msc{C})$, has no solutions to the equation $c_{V,V}=-id_{V\ot V}$ then $\msc{E}$ must be Tannakian.
\end{proof}

Following~\cite{etingof02}, we define the (braided) quasiexponent $\mrm{qexp}_\mrm{br}(\msc{C})$ of a braided tensor category $\msc{C}$ is the minimal integer $N$ such that the double braiding $(c^2_{V,W})^N$ is a unipotent automorphism, at arbitrary $V$ and $W$.  This number was shown to be finite at~\cite[Theorem.~4.1]{etingof02}.\footnote{In~\cite{etingof02}, Etingof shows precisely that the double braiding is unipotent at any given pair of objects.  However, unipotency at each pair of simples implies global unipotency.  Also, we slightly deviate from~\cite{etingof02}, where the quasiexponent is defined via the center $Z(\msc{C})$, which requires no braiding assumption on $\msc{C}$.  Whence our use of the term ``braided" quasiexponent.} 

\begin{lemma}
If $\mrm{qexp}_\mrm{br}(\msc{C})$ is odd then the M\"uger center of $\msc{C}$ is Tannakian.
\end{lemma}

\begin{proof}
For any non-Tannakian symmetric category $\msc{E}$ we have $\mrm{qexp}_\mrm{br}(\msc{E})=2$.  Furthermore, one can readily check that for any braided embedding $\msc{E}\to \msc{C}$ we have $\mrm{qexp}_\mrm{br}(\msc{E})\mid \mrm{qexp}_\mrm{br}(\msc{C})$.  Thus, if the quasiexponent of $\msc{C}$ is odd then $\msc{C}$ admits no braided embedding from a super-Tannakian category.
\end{proof}

To summarize.

\begin{corollary}\label{cor:braidedI}
Suppose $\msc{C}$ is braided and of finite type over a field of characteristic $0$.  Let $\msc{M}$ be an exact $\msc{C}$-module category and consider the dual $\msc{C}^\ast_{\msc{M}}$.  The category $\msc{C}^\ast_{\msc{M}}$ is of finite type provided any of the following properties hold for $\msc{C}$:
\begin{itemize}
\item $\C$ is odd-dimensional.
\item The quasiexponent $\mrm{qexp}_\mrm{br}(\msc{C})$ is odd.
\item There are no simples in $Z_2(\msc{C})$ with $c_{V,V}=-id_{V\ot V}$.
\item There are no simples in $\msc{C}$ with $c_{V,V}=-id_{V\ot V}$.
\item $Z_2(\C)$ is trivial, that is $\C$ is non-degenerate.
\item $Z_2(\C)$ is Tannakian.
\end{itemize}
\end{corollary}

\section{Finite type braided categories III: semisimple M\"uger center}
\label{sect:braidedIII}

We fix $k$ algebraically closed and of characteristic $0$.  We have seen previously that if $\msc{C}$ of finite type and braided with Tannakian M\"uger center, then all duals $\msc{C}^\ast_\msc{M}$ are also of finite type.  We extend this result to allow (more generally) for any semisimple M\"uger center.

\begin{theorem}\label{thm:braidedII}
Let $\msc{C}$ be a braided tensor category of finite type over an algebraically closed field of characteristic $0$. Suppose that the M\"uger center of $\msc{C}$ is semisimple.  Then, for any exact module category $\msc{M}$, the dual category $\msc{C}^\ast_\msc{M}$ is also of finite type and the Krull dimension is bounded as
\[
\Kdim\msc{C}^\ast_\msc{M}\leq 2 \Kdim\msc{C}.
\] 
\end{theorem}

We first discuss $G$-extensions, and minimal non-degenerate extensions for slightly degenerate categories.  We then employ some results of~\cite{davydovnikshychostrik13} to prove Theorem~\ref{thm:braidedII}.  After proving Theorem~\ref{thm:braidedII}, we briefly discuss means of determining semisimplicity of M\"uger centers.

\subsection{Cohomology and $G$-extensions}
\label{sect:Gext}

Fix $G$ a finite group.  A $G$-grading on a tensor category $\msc{D}$ is a choice of decomposition $\msc{D}=\oplus_{g\in G}\msc{D}_g$ such that $\msc{D}_g\ot\msc{D}_h\subset \msc{D}_{gh}$, where the components are full, orthogonal, $k$-linear subcategories of $\msc{D}$.  We call such a grading on $\msc{D}$ faithful if the $\msc{D}_g$ are nonzero for each $g\in G$.
\par

An embedding $i:\msc{C}\to \msc{D}$ of tensor categories is called a {\it $G$-extension} if $\msc{D}$ comes equipped with a faithful grading by $G$ under which $\msc{C}$ is identified with the neutral component $\msc{D}_1$ via $i$.

\begin{lemma}\label{lem:Gext}
Let $i:\msc{C}\to \msc{D}$ be a $G$-extension. Then $\msc{C}$ is of finite type if and only if $\msc{D}$ is of finite type and, in this case, the Krull dimensions agree.
\end{lemma}

\begin{proof}
Let $p:\msc{D}\to \msc{D}_1=\msc{C}$ be the $k$-linear projection onto the neutral component.  Since $\msc{D}$ decomposes as an abelian category $\msc{D}=\oplus_{g\in G}\msc{D}_g$, and since $\mbf{1}$ is an object in $\msc{C}$, we see that the inclusion
\begin{equation}\label{eq:103}
H^\b(\msc{C},pV)\overset{i}\to H^\b(\msc{D},pV)\to H^\b(\msc{D},V)
\end{equation}
is an isomorphism for each $V$ in $\msc{D}$.  Indeed, both of the maps in the composition are isomorphisms. In particular, the inclusion $i:\msc{C}\to \msc{D}$ identifies the algebras $H^\b(\msc{C},\1)$ and $H^\b(\msc{D},\1)$, and~\eqref{eq:103} is an isomorphism of $H^\b(\msc{C},\1)=H^\b(\msc{D},\1)$-modules.  So we see that $\msc{C}$ is of finite type if and only if $\msc{D}$ is of finite type, and that the Krull dimensions agree in this case.
\end{proof}

\begin{remark}
One can show furthermore that for any $G$-extension $\msc{C}\to\msc{D}$, the double $Z(\msc{C})$ is of finite type if and only if $Z(\msc{D})$ is of finite type.  Indeed, there is a $k$-linear decomposition $Z(\msc{D})=Z_1\oplus Z_1^\perp$, where $Z_1$ is the preimage of $\msc{C}$ along the forgetful functor $Z(\msc{D})\to \msc{D}$, and one uses~\cite[Corollary 3.7]{gelakinaidunikshych09} to obtain $Z(\msc{C})$ as a de-equivariantization of $Z_1$ (cf.~\cite[Theorem 2.4]{gelakisebbag}).
\end{remark}

\subsection{Minimal extensions and the fermionic grading~\cite{bghnprw} (cf.~\cite[Section\ 3]{gelakisebbag})}

Consider a braided tensor category $\msc{C}$ with M\"uger center isomorphic to $sVect$.  By a minimal non-degenerate extension of $\msc{C}$ we mean a braided embedding $\msc{C}\to \msc{D}$ for which $\msc{D}$ is non-degenerate and of Frobenius-Perron dimension $2 \opn{FPdim}(\msc{C})$.  The $2$ here comes from the Frobenius-Perron dimension of the degeneracy $sVec\cong Z_2(\msc{C})$.
\par

Fix $\msc{C}$ a braided tensor category with M\"uger center isomorphic to $sVect$ and $\msc{C}\to \msc{D}$ a minimal non-degenerate extension.  Let $f\in G(\msc{C})$ be the unique non-trivial simple in $Z_2(\msc{C})$.  Tensoring by $f$ is an automorphism of $\msc{D}$, and the double braiding $c^2_{f,-}$ is a natural automorphism of the functor $f\ot-:\msc{D}\to \msc{D}$.  We take $\mbf{c}\in \mrm{Aut}_{\mrm{Fun}}(id_\msc{D})$ to be the unique natural automorphism of the identity functor so that $f\ot \mbf{c}=c^2_{f,-}$.

\begin{lemma}
The automorphism $\mbf{c}$ is such that $\mbf{c}^2=1$.
\end{lemma}

\begin{proof}
Let $V$ be an arbitrary object in $\msc{D}$.  Since $f\ot f\cong \1$ we have
\[
\begin{array}{rl}
f^2\ot id_{V}=c^2_{f^2,V}&=(f\ot c_{V,f})(c_{V,f}\ot f)(c_{f,V}\ot f)(f\ot c_{f,V})\\
&=(f\ot c_{V,f})(c^2_{f,V}\ot f)(f\ot c_{f,V})\\
&=(f\ot c_{V,f})(f\ot \mbf{c}_V\ot f)(f\ot c_{f,V})\\
&=(f\ot c_{V,f})(f\ot c_{f,V})(f^{2}\ot \mbf{c}_V)\\
&=f^{2}\ot \mbf{c}_V^2.
\end{array}
\]
In the above calculation, we have employed MacLane's strictness theorem to ignore the associators. Since the above equations holds at all $V$ in $\msc{D}$, we find $\mbf{c}^2=1$.
\end{proof}

The global operators $(1\pm \mbf{c})$ satisfy
\[
(1-\mbf{c})(1+\mbf{c})=0,\ \ (1+\mbf{c})+(1-\mbf{c})=2,\ \ (1\pm\mbf{c})^2=2(1\pm\mbf{c}).
\]
Hence the global operators $p_{\pm 1}:=\frac{1}{2}(1\pm \mbf{c})$ provide orthogonal idempotents which sum to $1$.  We therefore get a $k$-linear splitting
\[
\msc{D}=\msc{D}_1\oplus \msc{D}_{-1},\ \ \text{where}\ \ \msc{D}_{\pm 1}=p_{\pm 1}\msc{D}.
\]

\begin{lemma}[{cf.~\cite{bghnprw}}]
The decomposition $\msc{D}=\msc{D}_1\oplus \msc{D}_{-1}$ is a faithful $\mbb{Z}/2\mbb{Z}$-grading.  In addition $\msc{C}\subset \msc{D}_1$.
\end{lemma}

\begin{proof}
The fact that $p_{\pm 1}$ are orthogonal idempotents implies that there are no maps, and no non-trivial extensions, between objects in $\msc{D}_1$ and objects in $\msc{D}_{-1}$.  So the decomposition is a decomposition of $\msc{D}$ as a $k$-linear category.  The fact that $\msc{D}_{a}\msc{D}_b=\msc{D}_{ab}$ follows from the braid relation.
\par

Since $f$ is in the M\"uger center of $\msc{C}$ we have $c^2_{f,-}|_\msc{C}=f\ot id$ and hence $\mbf{c}|_\msc{C}=id_\msc{C}$.  So $\msc{C}\subset \msc{D}_1$.  To see that the grading is faithful, i.e.\ that $\msc{D}_{-1}\neq 0$, we note that if $\msc{D}_{-1}$ vanished then $f$ would provide a non-trivial simple object in the M\"uger center of $\msc{D}$, which cannot occur since $\msc{D}$ is non-degenerate.
\end{proof}

We call the above defined $\mbb{Z}/2\mbb{Z}$-grading on the minimal non-degenerate extension $\msc{D}$ the {\it fermionic grading}.  The existence of such a grading in the fusion setting was established first in~\cite{bghnprw}.

\begin{lemma}[{\cite[Proposition\ 8.20]{etingofnikshychostrik05}}]\label{lem:235}
If $\msc{D}$ is faithfully graded by a finite group $G$ then $\opn{FPdim}(\msc{D})=|G|\opn{FPdim}(\msc{D}_1)$.
\end{lemma}

\begin{proof}
The proof is just as in~\cite{etingofnikshychostrik05}, except one uses the regular object of~\cite[Definition 6.1.6]{egno15} in place of the element $R$.
\end{proof}

\begin{lemma}\label{lem:Z2ferm}
The embedding $\msc{C}\to \msc{D}_1$ is an equivalence.  That is to say, any minimal non-degenerate extension $\msc{C}\to \msc{D}$ is a $\mbb{Z}/2\mbb{Z}$-extension of $\msc{C}$ with respect to the fermionic grading.
\end{lemma}

\begin{proof}
By Lemma~\ref{lem:235}, $\operatorname{FPdim}(\msc{D})=2\operatorname{FPdim}(\msc{D}_1)$.  By our minimality assumption we also have $\opn{FPdim}(\msc{D})=2\opn{FPdim}(\msc{C})$.  Hence $\opn{FPdim}(\msc{C})=\opn{FPdim}(\msc{D}_1)$.  Agreement of Frobenius-Perron dimension implies that the embedding $\msc{C}\to \msc{D}_1$ is an equivalence~\cite[Proposition\ 6.3.3]{egno15}.
\end{proof}

\subsection{The Drinfeld center of a slightly degenerate category}

Consider a braided tensor category $\msc{C}$ with semisimple M\"uger center $\msc{E}$. Let $Z(\msc{C},\msc{E})$ be the M\"uger centralizer of $\msc{E}$ in $Z(\msc{C})$.  To be clear, we embed $\msc{E}$ in $Z(\msc{C})$ via the map $\msc{C}\to Z(\msc{C})$ induced by the braiding, and we consider the full subcategory $Z(\msc{C},\msc{E})$ of all $X$ in $Z(\msc{C})$ such that $\gamma_{X,V}\gamma_{V,X}=id_{V\ot X}$ for every $V$ in $\msc{E}$, where $\gamma$ denotes the braiding on $Z(\msc{C})$.
\par

The following proposition is a straightforward generalization of~\cite[Corollary\ 4.4]{davydovnikshychostrik13}.

\begin{proposition}\label{prop:dno}
Let $\msc{C}$ be a braided tensor category with M\"uger center $\msc{E}$.  Then
\begin{enumerate}
\item[(i)] The canonical functor $F:\msc{C}\bt\msc{C}^\mrm{rev}\to Z(\msc{C})$ is a surjection onto $Z(\msc{C},\msc{E})$.
\item[(ii)] The inclusion $\msc{E}\to Z(\msc{C},\msc{E})$ is an equivalence onto the M\"uger center of $Z(\msc{C},\msc{E})$.
\item[(iii)] When $\msc{E}\simeq sVect$, the surjection $\msc{C}\bt\msc{C}^\mrm{rev}\to Z(\msc{C},sVect)$ extends to an exact sequence $\mrm{rep}(\mbb{Z}/2\mbb{Z})\to \msc{C}\bt\msc{C}^\mrm{rev}\to Z(\msc{C},sVect)$.
\end{enumerate}
\end{proposition}

\begin{proof}
Let $\msc{C}_+$ and $\msc{C}_-$ denote the images of $\msc{C}$ and $\msc{C}^\mrm{rev}$ in $Z(\msc{C})$ under the embeddings given by the braidings.  For a subcategory $\msc{D}$, we let $\msc{D}'$ denote its centralizer.
\par

(i) \& (ii) The image of $F$ is the tensor subcategory $\msc{C}_+\vee\msc{C}_-$ generated by $\msc{C}_+$ and $\msc{C}_-$.  We note that the centralizer of $\msc{C}_\pm$ in $Z(\msc{C})$ is $\msc{C}_\mp$, and hence $(\msc{C}_+\vee\msc{C}_-)'\subset \msc{C}_+\cap\msc{C}_-=\msc{E}$.  The opposite containment follows from the fact that $\msc{E}$ centralizes both $\msc{C}_+$ and $\msc{C}_-$.  Hence $(\msc{C}_+\vee\msc{C}_-)'=\msc{E}$ and, since $Z(\msc{C})$ is non-degenerate, we have
\[
Z(\msc{C},\msc{E})=\msc{E}'=(\msc{C}_+\vee\msc{C}_-)''=\msc{C}_+\vee\msc{C}_-,
\]
by \cite[Theorem\ 4.9]{shimizu18}.  It follows that $F$ provides a surjective map onto $Z(\msc{C},\msc{E})$.
\par

(iii) Let $f$ denote the fermion in $sVect\subset \msc{C}$ and let $\zeta$ denote the unique non-trivial simple in $\mrm{rep}(\mbb{Z}/2\mbb{Z})$.  We define a braided embedding $i:\mrm{rep}(\mbb{Z}/2\mbb{Z})\to\msc{C}\bt\msc{C}^\mrm{rev}$ by taking $i(\zeta)=f\bt f$.  Since $\zeta\ot\zeta\cong \1$, $F(\mrm{rep}(\mbb{Z}/2\mbb{Z}))$ is the trivial subcategory in $Z(\msc{C},\msc{E})$.  Whence we have the sequence
\begin{equation}\label{eq:1708}
\mrm{rep}(\mbb{Z}/2\mbb{Z})\overset{i}\to \msc{C}\bt\msc{C}^\mrm{rev}\overset{F}\to Z(\msc{C},\msc{E}),
\end{equation}
with $F$ surjective and $Fi$ factoring through the fiber functor $\mrm{rep}(\mbb{Z}/2\mbb{Z})\to Vect$.
\par

Now, by~\cite[Lemma\ 4.8]{shimizu18} we have
\[
\opn{FPdim}(Z(\msc{C},\msc{E}))=\frac{\opn{FPdim}(\msc{C}_+)\opn{FPdim}(\msc{C}_-)}{\opn{FPdim}(sVect)}=\frac{\opn{FPdim}(\msc{C}\bt\msc{C}^\mrm{rev})}{\opn{FPdim}(\mrm{rep}(\mbb{Z}/2\mbb{Z}))}.
\]
By~\cite[Theorem\ 3.4]{etingofgelaki17} it follows that the sequence~\eqref{eq:1708} is exact.
\end{proof}

For our purposes, one can replace (iii) with the equally useful statement that $F$ induces an equivalence $\msc{C}\bt_{sVect}\msc{C}^\mrm{rev}\cong Z(\msc{C},sVect)$~\cite{davydovnikshychostrik13}.  Here we define the product over $sVect$ as the de-equivariantization of $\msc{C}\bt\msc{C}$ by the Tannakian subcategory $\mrm{rep}(\mbb{Z}/2\mbb{Z})$.

\begin{corollary}\label{cor:1679}
If $\msc{C}$ is of finite type, and the M\"uger center of $\msc{C}$ is equivalent to $sVect$, then $Z(\msc{C},sVect)$ is also of finite type and $\Kdim Z(\msc{C},sVect)=2 \Kdim\msc{C}$.
\end{corollary}

\begin{proof}
The surjective tensor functor $F:\msc{C}\bt\msc{C}^\mrm{rev}\to Z(\msc{C},sVect)$ implies that $Z(\msc{C},sVect)$ is of finite type, by Proposition~\ref{prop:fgcat}.  Also, the spectral sequence of Proposition~\ref{prop:ss} gives 
\[
H^\b(\msc{C}\bt\msc{C}^\mrm{rev},\1)=H^\b(Z(\msc{C},sVect),\1)^{\mbb{Z}/2\mbb{Z}},
\]
from which we deduce the Krull dimensions.
\end{proof}

\begin{lemma}\label{lem:1728}
Suppose $\msc{C}$ is a braided tensor category of finite type with M\"uger center equivalent to $sVect$. Then the Drinfeld center $Z(\msc{C})$ is of finite type and $\Kdim Z(\msc{C})=2 \Kdim \msc{C}$. 
\end{lemma}

\begin{proof}
By Proposition~\ref{prop:dno} we have that $Z(\msc{C},sVect)$ has M\"uger center $sVect$.  As we saw in the proof of Proposition~\ref{prop:dno}, we also have
\[
\opn{FPdim}(Z(\msc{C},sVect))=\frac{\opn{FPdim}(\msc{C})^2}{2}.
\]
Thus the inclusion $Z(\msc{C},sVect)\to Z(\msc{C})$ is a minimal non-degenerate extension and, by Lemma~\ref{lem:Z2ferm}, $Z(\msc{C})$ is a $\mbb{Z}/2\mbb{Z}$-extension of $Z(\msc{C},sVect)$.  By Lemma~\ref{lem:Gext} and Corollary~\ref{cor:1679} it follows that $Z(\msc{C})$ is of finite type and of the proposed Krull dimension.
\end{proof}

\subsection{Proof of Theorem~\ref{thm:braidedII}}
\label{sect:266}

\begin{proof}[Proof of Theorem~\ref{thm:braidedII}]
By Corollary \ref{cor:Z2C}, it suffices to show that the Drinfeld center $Z(\msc{C})$ is of finite type and that $\Kdim Z(\msc{C})=2\Kdim\msc{C}$.  By Theorem~\ref{thm:ZCvZC_G}, we may de-equivariantize by the maximal Tannakian part of the M\"uger center to assume $Z_2(\msc{C})\subseteq sVect$.  The result now follows by Proposition~\ref{prop:Cduals} (non-degenerate case) and Lemma~\ref{lem:1728} (slightly degenerate case).
\end{proof}

\subsection{Semisimplicity of the M\"uger center}
\label{sect:ssE}

In general, detecting semisimplicity of the M\"uger center of a given braided tensor category is a non-trivial problem.  Specifically, it can be difficult to tell if the M\"uger center is semisimple but not Tannakian (cf.~Corollary~\ref{cor:braidedI}).  We provide a minimal discussion of the topic here.

Since M\"uger central semisimple objects in a finite symmetric tensor category form a fusion subcategory isomorphic to $\rep (G,R_z)$, for some central element $z$ of order at most $2$, we understand that the quantum dimension of a M\"uger central simple $V$ is $\pm \operatorname{FPdim}(V)$.  This is due to the fact that all simples appear as irreducible summands of
\[
k[G]=k[G]\ot_{k[\mbb{Z}/2\mbb{Z}]}k[\mbb{Z}/2\mbb{Z}]=(k[G]\ot_{k[\mbb{Z}/2\mbb{Z}]}k)\oplus (k[G]\ot_{k[\mbb{Z}/2\mbb{Z}]}sgn),
\]
where $\mbb{Z}/2\mbb{Z}$ embeds in $G$ via $z$ and $sgn$ is the unique nontrivial simple representation for $\mbb{Z}/2\mbb{Z}$.  We call a simple $V$ positive (resp.\ negative) if $\operatorname{qdim}(V)>0$ (resp.\ $<0$).  We let $\operatorname{Irr}(\msc{C})_\pm$ denote the respective collections of positive or negative simples in $\msc{C}$.

\begin{lemma}\label{lem:4}
Let $\msc{C}$ be a weakly integral and braided tensor category, and let $\msc{E}$ denote its M\"uger center.  Let $\msc{E}_+$ denote the fusion subcategory generated by the positive simples in $\msc{E}$.  If $4\FPdim(\msc{E}_+)\nmid \opn{FPdim}(\msc{C})$, then the M\"uger center of $\msc{C}$ is semisimple.  In particular, if $4\nmid \FPdim(\msc{C})$ then the M\"uger center of $\msc{C}$ is semisimple.
\end{lemma}

\begin{proof}
Suppose that the M\"uger center $\msc{E}$ is not semisimple. Then $\msc{E}$ is the representation category of Hopf algebra of the form $\mrm{Wedge}(V)\rtimes G$, where $G$ is a finite group with a specified order $2$ element $z$ and $V$ nonvanishing, by Deligne~\cite{deligne90,deligne02} (see also~\cite{andrusetingofgelaki01}).  We have
\[
\FPdim(\msc{E})=|G|\dim(\mrm{Wedge}(V))=2\FPdim(\msc{E}_+)2^{\dim(V)}.
\]
If $\msc{C}$ is weakly integral this implies $4\FPdim(\msc{E}_+)\mid \FPdim(\msc{E})\mid \opn{FPdim}(\msc{C})$.
\end{proof}

Note that one need only find the M\"uger central simples in $\msc{C}$ in order to calculate $\FPdim(\msc{E}_+)$.  For it is simply given as the sum
\[
\FPdim(\msc{E}_+)=\sum_{V\in \operatorname{Irr}(\msc{E})_+}\FPdim(V)^2.
\]

\begin{lemma}\label{lem:1858}
A symmetric category $\msc{E}$ is {\it not} semisimple if and only if any of the following occur:
\begin{itemize}
\item $\operatorname{FPdim}(\msc{E})>\sum_{W\in \operatorname{Irr}(\msc{E})}\FPdim(W)^2$.
\item $\operatorname{FPdim}(\msc{E})> 2\sum_{V\in \operatorname{Irr}(\msc{E})_+}\FPdim(V)^2$.
\item $\Ext^1_\msc{E}(V,W)\neq 0$ for some positive simple $V$ and negative simple $W$.
\item $\Ext^1_\msc{E}(\1,W)\neq 0$ for some negative simple $W$.
\end{itemize}
\end{lemma}

\begin{proof}
The first point follows from the fact that the additive subcategory $\bar{\msc{E}}$ generated by the simples in $\msc{E}$ forms a fusion subcategory, which is of Frobenius-Perron dimension $\sum_{W\in \operatorname{Irr}(\msc{E})}\FPdim(W)^2$.  The second point follows from the fact that $\FPdim(\bar{\msc{E}})=2\FPdim(\msc{E}_+)$.  For the third point note that the tensor subcategory generated by extensions of the positive simples has no negative simples, is Tannakian by Corollary~\ref{cor:braidedI}, and hence fusion.  Therefore all extensions between positive simples vanish.  Hence $\msc{E}$ is non-semisimple if and only there is some extension from a positive simple to a negative simple.  The final point follows by the formula $\Ext^1_\msc{E}(V,W)=\Ext^1_\msc{E}(\1,W\ot V^\ast)$.
\end{proof}


Of course, vanishing of M\"uger central extensions is seen most readily if extensions already vanish in $\msc{C}$.  Such vanishing occurs for degenerate quantum groups at even roots of unity, for example.  For instance, if we consider quantum $\mrm{PSL}(2)$ at $q=i$, then $\rep\mrm{PSL}(2)_i$ admits a minimal de-equivariantization to finite braided tensor category $dE:\rep\mrm{PSL}(2)_i\to\msc{C}_i$ with degeneracy $Z_2(\msc{C}_i)=sVect$~\cite{negron}.  In this case, extensions between representations in $Z_2(\msc{C}_i)$ already vanish in $\msc{C}_i$.

\begin{corollary}
Suppose that $\Ext^1_\msc{C}(V,W)$ vanishes for all positive simple $V$ and negative simple $W$ in the M\"uger center of $\msc{C}$.  Then the M\"uger center of $\msc{C}$ is semisimple.  Similarly, the M\"uger center is semisimple if $\Ext^1_\msc{C}(\1,W)$ vanishes for all negative central $W$.
\end{corollary}

For some examples of finite tensor categories with varying M\"uger center one can see~\cite{gainutdinovlentnerohrmann18}, where the authors consider ``pointed" quasi-Hopf algebras $A$ with ``grouplikes" given by a cocycle-twisted group algebra.  More generally, for any braided tensor category $\msc{K}$ (such as those given by abelian groups with a choice of quadratic form~\cite{joyalstreet93}, or those given by choices of braidings on $\rep G$ for possibly non-abelian $G$~\cite{davydov97,natale06}) and choice of Hopf algebra $R$ in $\msc{K}$, one can construct the center $\operatorname{YD}(\msc{K})^R_R$ of $\rep(R)$ relative to $\msc{K}$~\cite[Theorem 6.2]{shimizu18}.  This is a finite tensor category with M\"uger center equal to that of $\msc{K}$.  In particular, the M\"uger center of $\operatorname{YD}(\msc{K})^R_R$ is semisimple when $\msc{K}$ is a braided fusion category.  One can also consider, of course, products $\msc{C}\boxtimes \msc{W}$ of a non-degenerate braided category $\msc{C}$ with an arbitrary braided fusion category $\msc{W}$.

\appendix

\section{Dynamical twists and module categories}
\label{sect:A}

We discuss how one can understand dynamical twists, and dynamical cocycle twists, via module categories.  This material will certainly be unsurprising to experts.  One can see a light version of the below discussion in~\cite[Section 4.4]{ostrik03}, for example.  The material of this appendix is used in Section~\ref{sect:uqT}.

We first recall, briefly, relations between Hopf algebroids, weak Hopf algebras, and (finite) tensor categories.

\subsection{Hopf algebroids, weak Hopf algebras, and tensor categories}

We only give here a reminder of Hopf algebroids and weak Hopf algebras.  We refer the reader to~\cite{szlachanyi00,xu01,etingofnikshych01} for precise definitions.
\par

A (left) bialgebroid over a base algebra $R$ is an algebra $A$ equipped with a structure map $R\ot R^{op}\to A$, coassociative comultiplication $\Delta:A\to A\ot_R A$, and (left) counit $\epsilon:A\to R$.  The comultiplication is required to be an algebra map, although one needs to place some restrictions on the image of $\Delta$ in $A\ot_R A$ in order for this to make sense.  The structure $(\Delta,\epsilon)$ on $R\ot R^{op}\to A$ is equivalent to a choice of monoidal structure on $\rep(A)$ so that the forgetful functor $\rep(A)\to \mrm{bimod}(R)$ is strict monoidal.
\par

A Hopf algebroid $A$ over $R$ is a bialgebroid for which $\rep(A)$ is rigid.  Since right and left duals are unique, and preserved under monoidal functors, one can conclude that for any Hopf algebroid $A$ each object in $\rep(A)$ is projective over $R$~\cite[Exercise 2.10.16]{egno15}.  Of course, the tensor product $A\ot A'$ of Hopf algebroids over $R$ and $R'$ respectively is a Hopf algebroid over $R\ot R'$.
\par

We have the following basic fact, which is apparent from the work of Szlachanyi~\cite{szlachanyi00}.

\begin{lemma}\label{lem:CasA}
Any tensor category $\msc{C}$ admits a tensor equivalence $\msc{C}\overset{\sim}\to \mrm{rep}(A)$ for a Hopf algebroid $A$.
\end{lemma}

\begin{proof}[Sketch proof]
One considers a finite dimensional algebra $R$ with an exact $\msc{C}$-module structure on $\msc{M}=\rep(R)$ then uses the corresponding representation $\rho:\msc{C}\to \End(\msc{M})=\mrm{bimod}(R)$ to construct the desired algebroid $A$ over $R$, as in~\cite[Theorem 1.8]{szlachanyi00}.  For example, one can take $R$ such that $\msc{M}=\msc{C}\cong \rep(R)$ and consider the regular representation for $\msc{C}$.
\end{proof}

When $A$ is a Hopf algebroid over a separable base $R$, one can also use a bimodule splitting $R\to R\ot R$ of the multiplication map to lift the comultiplication for $A$ to a map $\tilde{\Delta}:A\to A\ot A$.  When the unit in $\rep(A)$ is simple, with $\End_A(\mbf{1})=k$, we furthermore get a canonical map $\tilde{\epsilon}:A\to k$.  The resulting structure $(A,\tilde{\Delta},\tilde{\epsilon})$ is a {\it weak Hopf algebra} (see~\cite[Section 4]{ostrik03}).

\subsection{Dynamical twists and module categories}
\label{sect:dynamM}

We refer the reader to~\cite{mombelli07,etingofnikshych01} for basic information regarding dynamical twists.  We present some relationships between dynamical twists, module categories, and dynamical quantum group constructions from~\cite{xu01,etingofnikshych01}.  We consider a finite abelian group $\Lambda$ and suppose $\mrm{char}(k)\nmid|\Lambda|$.
\par

Suppose $A$ is a Hopf algebra and $\Lambda$ is an abelian subgroup in the group of grouplikes $G(A)$.  Let $J:\Lambda^\vee\to A\ot A$ be a dynamical twist.  In particular, $J$ is a map into the $\Lambda$-invariants $(A\ot A)^\Lambda$ under the (diagonal) adjoint action which solves a parameter dependent dual cocycle condition.
\par

The forgetful functor $\rep(A)\to \rep(\Lambda)$ induces an exact module category structure on $\rep(\Lambda)$ and we use $J$ to perturb the associativity and hence produce a new module category structure $\msc{M}(J)$ on $\rep(\Lambda)$.  Directly, the associativity is given by
\[
\mrm{assoc}_{\msc{M}(J)}:X\ot (Y\ot V)\to (X\ot Y)\ot V,\ \ x\ot y\ot v\mapsto \sum_{\chi\in \Lambda^\vee}(J(\chi)\ot P_\chi)(x\ot y\ot v),
\]
where $P_\chi$ is the usual idempotent $|\Lambda|^{-1}\sum_{\lambda\in \Lambda}\chi(\lambda)\lambda^{-1}$.  From the action of $\mrm{rep}(A)$ on $\msc{M}(J)$ we get a faithful monoidal embedding
\[
F(J):\rep(A)\to \mrm{bimod}(k[\Lambda]).
\]
We then follow Szlachanyi~\cite[Theorem 1.8]{szlachanyi00} to produce a Hopf algebroid $A^J$ over $k[\Lambda]$ equipped with an equivalence to $\rep(A)$ over $\mrm{bimod}(k[\Lambda])$,
\begin{equation}\label{eq:1001}
\xymatrix{
\rep(A)\ar[dr]_{F(J)}\ar[rr]^{\sim}_{\mrm{Sz}(J)}& &\rep(A^J)\ar[dl]^{\mrm{Forget}}\\
 & \mrm{bimod}(k[\Lambda]) &.
}
\end{equation}
As an algebra $A^J=A\ot\End_k(k[\Lambda])$.\footnote{We are actually using the {\it opposite} algebroid to the one constructed in~\cite{szlachanyi00}.}  More directly, $A^J$ is given in~\cite{szlachanyi00} as the endomorphism algebra $\End_A(A\ot\O(\Lambda))^{op}$ which is canonically identified with $A\ot \End_k(k[\Lambda])$ via the isomorphism
\[
A\ot \End_k(k[\Lambda])\to \End_A(A\ot \O(\Lambda))^{op},\ \ a\ot f\mapsto (-\cdot a)\ot f^\ast.
\]

Since the base $k[\Lambda]$ is separable, $A^J$ is identified with a weak Hopf algebra by splitting the multiplication for $k[\Lambda]$, as described in~\cite[Proof of Theorem 4.1]{ostrik03}.

\begin{lemma}\label{lem:1015}
For a dynamical twist $J:\Lambda^\vee\to A\ot A$, the weak Hopf algebra $A^J$ constructed from the corresponding exact module category $\msc{M}(J)$ is equal to the Xu style twisted weak Hopf algebra, as constructed in~\cite[Proposition 4.2.4]{etingofnikshych01}.
\end{lemma}

\begin{proof}[Sketch proof]
The constructions of~\cite{szlachanyi00} and~\cite{etingofnikshych01} are both explicit, and both of the proposed twisted weak Hopf algebras are equal to $A\ot\End_k(k[\Lambda])$ as algebras.  One simply writes down the coproduct for $A^J$ and the Xu style twitsted algebra and observed directly that they are equal.  Specifically, both comultiplications are given by the formula
\[
\Delta^J(a\ot E_{\mu,\nu})=\sum_{\tau,\sigma}(1\ot P_{\sigma})J^{-1}(\mu)\Delta(a)J(\nu)(E_{\sigma\mu,\nu\tau}\ot P_\tau E_{\mu,\nu}),
\]
where $E_{\mu,\nu}\in \End_k(k[\Lambda])$ is the elementary matrix which maps idempotents as $E_{\mu,\nu}(P_\tau)=\delta_{\nu,\tau}P_\mu$, and the sum is over all $\tau,\sigma\in \Lambda^\vee$.
\end{proof}

Recall that the vector space dual of a weak Hopf algebra is another weak Hopf algebra.  For the twisted weak Hopf algebra $A^J$ we write $A^\ast_J$ for the vector space dual.  We view $A^\ast_J$ as a ``dynamical cocycle twist" of $A^\ast$.  The dual $A^\ast_J$ is opposite to the Etingof-Varchenko style dynamical twisted algebra for the pair $(A,J)$~\cite[Theorem 4.3.1]{etingofnikshych01} (cf.~\cite{etingofvarchenko99}).

\begin{lemma}\label{lem:1025}
For $J$ and $A$ as in Lemma~\ref{lem:1015}, there is an equivalence of tensor categories $\mrm{Sz}^\ast(J):\rep(A^\ast_J)^{\mrm{cop}}\overset{\sim}\to \rep(A)^\ast_{\msc{M}(J)}$.
\end{lemma}

\begin{proof}
This is an immediate consequence of the diagram~\eqref{eq:1001} and~\cite[Theorem 4.2]{ostrik03}.
\end{proof}

\subsection{Constant dynamical twists}

In the case of a constant twist $J:\Lambda^\vee\to A$, with constant value $J^c$, the module category $\msc{M}(J)$ is $\mrm{rep}(\Lambda)$ with the constant associator given by multiplying by $J^c$
\[
\mrm{assoc}(J):X\ot (Y\ot V)\to (X\ot Y)\ot V,\ \ x\ot y\ot v\mapsto J^c_{12}(x\ot y\ot v).
\]
We note that for such a constant twist the value $J^c\in A\ot A$ lies in the $\Lambda$-invariants $(A\ot A)^\Lambda$, under the adjoint action.
\par

We have the left and right translation actions of $\Lambda$ on the dual $A^\ast$.  Specifically, for $\lambda\in \Lambda$ we act by the algebra automorphisms $\lambda\cdot f=\lambda^{-1}(f_1)f_2$ and $f\cdot \lambda=f_1\lambda^{-1}(f_2)$.  Restricting along these automorphisms gives an action of $\Lambda^e=\Lambda\times \Lambda^{op}$ on $\mrm{rep}(A^\ast)$, and we may take the ($k$-linear) equivariantization $\mrm{rep}(A^\ast)^{\Lambda^e}$.  Similarly, for the $\Lambda$-invariant twist $J^c$ we still have $k[\Lambda]\subset A^{J^c}$ and may take the equivariantization $\mrm{rep}(A^\ast_{J^c})^{\Lambda^e}$.  Algebraically, the equivariantization is the category of representations over the smash product $A^\ast_{J^c}\rtimes(\Lambda^e)$.
\par

We note that $\mrm{rep}(A^\ast_{J^c})^{\Lambda\times\Lambda^{op}}$ is {\it not} a tensor category under the usual product $\ot_k$, as the translation actions of $\Lambda$ on $A^\ast_{J^c}$ are not actions by Hopf automorphisms.  Via the algebra projection $A^\ast_{J^c}\to \O(\Lambda)$, and regular left action of $\O(\Lambda)$ on $k[\Lambda]$, we see that $k[\Lambda]$ is an object in $\mrm{rep}(A^\ast_{J^c})$.  We have the canonical equivariant structure on $k[\Lambda]$ given by the regular left and right actions of $\Lambda$.  The following is proved in~\cite[Proposition 1.23]{andruskiewitschmombelli07}.

\begin{proposition}[\cite{andruskiewitschmombelli07}]\label{prop:const}
Suppose that $J:\Lambda^\vee\to A$ is a constant dynamical twist for a Hopf algebra $A$, with constant value $J^c\in (A\ot A)^\Lambda$.  Then there is a $k$-linear equivalence
\[
F:\mrm{rep}(A)^\ast_{\msc{M}(J)}\overset{\sim}\to \mrm{rep}(A^\ast_{J^c})^{\Lambda\times\Lambda^{op}}
\]
under which the unit is sent to the $A^\ast_{J^c}$-representation $k[\Lambda]$, with the above equivariant structure.
\end{proposition}

One can deduce from the proof of~\cite{andruskiewitschmombelli07} that there is a canonical Hopf algebroid structure on the smash product $A^\ast_{J^c}\rtimes (\Lambda^e)$ so that the above equivalence is an equivalence of tensor categories.

\bibliographystyle{abbrv}

\end{document}